\documentclass{article}
\date{}
\usepackage{a4wide}

\usepackage{amsfonts}
\usepackage{amsmath}
\usepackage{amsthm}

\usepackage{amssymb}
\usepackage{mathtools}
\usepackage[T1]{fontenc}
\usepackage{graphicx}

\usepackage{xcolor}
\usepackage{hyperref}
\hypersetup{
    colorlinks,%
    citecolor=black,%
    filecolor=black,%
    linkcolor=black,%
    urlcolor=black
}

\newtheorem{theorem}{Theorem}[section]
\newtheorem{xmpl}[theorem]{Example}
\newtheorem{rmrk}[theorem]{Remark}
\newtheorem{prpstn}[theorem]{Proposition}
\newtheorem{crllr}[theorem]{Corollary}
\newtheorem{lmm}[theorem]{Lemma}

\newcommand{\Dc}{\mathcal{D}}
\newcommand{\Sc}{\mathcal{S}}
\newcommand{\Tc}{\mathcal{T}}
\newcommand{\Rc}{\mathcal{R}}

 \newcommand{\xR}{\mathbb{R}}

\def\xLtwo{{{\rm L}^2}}
\def\xHone{{{\rm H}^1}}

\usepackage{subfigure}
\usepackage{algorithmic}
\usepackage{algorithm}

\title{A tensor approximation method based on ideal minimal residual
formulations for the solution of high-dimensional problems
\thanks{This work is supported by the ANR (French National Research Agency, grants ANR-2010- COSI-006 and ANR-2010-BLAN-0904)}}
\author{Marie Billaud-Friess\footnotemark[2]
  \and Anthony Nouy\footnotemark[2]\ \footnotemark[3]  \and Olivier Zahm\footnotemark[2]}

\begin{document}

\maketitle

 \renewcommand{\thefootnote}{\fnsymbol{footnote}}
\footnotetext[2]{Ecole Centrale de Nantes, GeM UMR CNRS 6183, LUNAM Universit\'e, France.}
\footnotetext[3]{Corresponding author (anthony.nouy@ec-nantes.fr).}

\renewcommand{\thefootnote}{\arabic{footnote}}

\begin{abstract}  In this paper, we propose a method for the approximation of the solution of high-dimensional weakly coercive problems formulated in tensor spaces using low-rank approximation formats. The method can be seen as a perturbation of a minimal residual method with a measure of the residual corresponding to the error in a specified solution norm.  {The residual norm can be designed such that the resulting low-rank approximations are optimal with respect to particular norms of interest, thus allowing to take into account a particular objective in the definition of reduced order approximations of high-dimensional problems}. {We introduce and analyze an iterative algorithm that is able to provide an approximation of the optimal approximation of the solution in a given low-rank subset, without any a priori information on this solution.} We also introduce a weak greedy algorithm which  uses this perturbed minimal residual method for the computation of successive greedy corrections in small tensor subsets. We prove its convergence under some conditions on the parameters of the algorithm. 
The proposed numerical method is applied to the solution of a stochastic partial differential equation which is discretized using standard Galerkin methods in tensor product spaces. 
\end{abstract}

\pagestyle{myheadings}
\thispagestyle{plain}
\markboth{Marie Billaud-Friess, Anthony Nouy and Olivier Zahm}{Tensor approximation based on ideal minimal residual formulations}

\section*{Introduction}

Low-rank tensor approximation methods are receiving growing attention in computational science for the numerical solution of high-dimensional problems formulated in tensor spaces (see the recent surveys \cite{Kolda2009,Chinesta2011,Khoromskij2012,Grasedyck2013} and monograph \cite{Hackbusch2012}). Typical problems include the solution of high-dimensional partial differential equations arising in stochastic calculus, or the solution of stochastic or parametric partial differential equations using functional approaches, where functions of multiple (random) parameters have to be approximated.
These problems take the general form
\begin{align}
A(u)=b,\quad u\in X= X_{1}\otimes \hdots \otimes X_{d},\label{eq:intro_problem}
\end{align}
where $A$ is an operator defined on the tensor space $X$. Low-rank tensor methods then consist in searching an approximation of the solution $u$ in a low-dimensional subset $\Sc_{X}$ of elements of $X$ of the form
\begin{align} 
\sum_{i_{1}} \hdots \sum_{i_{d}} \alpha_{i_1 \hdots i_{d}} w_{i_{1}}^{1}\otimes \hdots \otimes w_{i_{d}}^{d}, \; w^{\mu}_{i_{\mu}} \in X_{\mu},\end{align}
where the set of coefficients $(\alpha_{i_1 \hdots i_{d}})$ possesses some specific structure. 
Classical low-rank tensor subsets include canonical tensors, Tucker tensors, 
{Tensor Train tensors \cite{Oseledets2011,Holtz2012manifolds},} Hierarchical Tucker tensors \cite{Hackbusch2009} or more general tree-based Hierarchical Tucker tensors \cite{Falco2013}. 
{In practice, many tensors arising in applications are observed to be efficiently approximable by elements of the mentioned subsets.} Low-rank approximation methods are closely related to a priori model reduction methods in that they provide approximate representations of the solution on low-dimensional reduced bases $\{w_{i_{1}}^{1}\otimes \hdots \otimes w_{i_{d}}^{d}\}$ that are not selected a priori. 
\\

The best approximation of $u\in X$ in a given low-rank tensor subset $\Sc_{X}$ with respect to a particular norm $\Vert\cdot\Vert_{X}$ in $X$ is the solution of
\begin{align} 
\min_{v\in \Sc_X} \Vert u-v \Vert_X. \label{eq:intro_approximation_SX}
\end{align}
{Low-rank tensor  subsets are neither linear subspaces nor convex sets.} However, they usually satisfy topological properties that make the above best approximation problem meaningful and allows the application of standard optimization algorithms \cite{Rohwedder2013,Espig2012,Uschmajew2012}. 
Of course, in the context of the solution of high-dimensional problems, the solution $u$ of problem \eqref{eq:intro_problem}
is not available, and the best approximation problem \eqref{eq:intro_approximation_SX} cannot be solved
directly. Tensor approximation methods then typically rely on the definition of approximations based on the residual
of equation \eqref{eq:intro_problem}, which is a computable
quantity.  Different strategies have been proposed for the construction of low-rank approximations of the solution of equations in tensor format. The first family of methods consists in using classical iterative
algorithms  for linear  or nonlinear systems of equations  with low-rank tensor algebra (using low-rank tensor compression) for standard algebraic operations \cite{Kressner2011,Khoromskij2011,Matthies2012,Ballani2013}.
The second family of methods consists in directly computing an approximation of $u$ in $\Sc_X$ by minimizing some residual norm \cite{Beylkin2005,Nouy2007,Doostan2009}: 
\begin{align}
\min_{v\in \Sc_X}  \Vert Av-b \Vert_{\star}. \label{eq:intro_minres_SX}
\end{align}
In the context of approximation, where one is interested in obtaining an approximation with a given precision rather than obtaining the best low-rank approximation, 
constructive greedy algorithms have been proposed that consist in computing successive corrections in a small low-rank tensor subset, typically the set of rank-one tensors \cite{Ladeveze1999,Ammar2006,Nouy2007}. These greedy algorithms have been analyzed in several papers  \cite{Ammar2010, Cances2011,Falco2011,Falco2012-pgd,Cances2012,Figueroa2012} and a series of improved algorithms have been introduced in order to increase the quality of suboptimal greedy constructions \cite{Nouy2010,Nouy2010-time,Ladeveze2010,Falco2012-pgd,Giraldi2013}. 
\\

Although minimal residual based approaches are well founded, they generally provide low-rank approximations that can be very far from optimal approximations with respect to the natural norm $\Vert\cdot\Vert_{X}$, at least when using usual measures of the residual.  
If we are interested in obtaining an optimal approximation with respect
to the norm $\Vert\cdot\Vert_X$, e.g. taking into account some particular quantity of interest, an ideal approach would be to define
the residual norm such that
$$
\Vert Av-b \Vert_{\star} = \Vert u-v\Vert_X,
$$
where $\Vert \cdot\Vert_X$ is the desired solution norm{, that corresponds to solve an ideally conditioned problem}. Minimizing the residual norm would therefore be equivalent to solving the best approximation problem \eqref{eq:intro_approximation_SX}. However, the
computation of such a residual norm is in general equivalent to the
solution of the initial problem \eqref{eq:intro_problem}.
\\

In this paper,
we propose a method for the approximation of the
ideal approach. This method applies to a general class of weakly coercive problems. It relies on the use
of approximations $r_{\delta}(v)$ of the residual $r(v)=Av-b$ such that $\Vert r_{\delta}(v)\Vert_{\star}$ approximates the ideal residual norm $\Vert r(v) \Vert_{\star} = \Vert u-v \Vert_{X}$. {The resulting method allows for the construction of low-rank tensor approximations which are quasi-optimal  with respect to a norm $\Vert\cdot\Vert_{X}$ that can be designed according to some quantity of interest.}
We first introduce and analyze an algorithm for minimizing the approximate residual norm $\Vert r_{\delta}(v) \Vert_{\star}$  in a given subset $\Sc_{X}$. This algorithm can be seen as an extension of the algorithms introduced  in \cite{Cohen2012,Dahmen2012} to the context of nonlinear approximation in subsets $\Sc_{X}$. 
{ It consists in a perturbation of a gradient algorithm for minimizing  in $\Sc_{X}$ the ideal residual norm 
$\Vert r(v) \Vert_{\star}$,  using approximations $r_{\delta}(v)$ of the residual $r(v)$. An ideal algorithm would consist in computing an approximation $r_{\delta}(v)$ such that 
\begin{equation}
(1-\delta) \Vert u-v \Vert_{X} \le \Vert r_{\delta}(v) \Vert_{\star} \le (1+\delta) \Vert
u-v\Vert_X, \label{eq:intro_delta}
\end{equation}
for some precision $\delta$, that requires the use of guaranteed error estimators.} {In the present paper, \eqref{eq:intro_delta} is not exactly satisfied since we only use heuristic error estimates.  However, these estimates seem to provide an acceptable measure of the error for the considered applications}.  The resulting algorithm can be interpreted as a preconditioned gradient algorithm with an implicit preconditioner that approximates the ideal preconditioner. 
Also, we propose a weak greedy algorithm for the adaptive construction of
an approximation of the solution of problem \eqref{eq:intro_problem}, using the perturbed ideal minimal residual approach for the computation of greedy corrections. A convergence proof is provided under some conditions on the parameters of the algorithm. \\

The outline of the paper is as follows. In section \ref{sec:weakly_coercive_problem}, we introduce a functional framework for weakly coercive problems. In section  \ref{sec:approx_tensors}, we briefly recall some definitions and basic properties of tensor spaces and low-rank tensor subsets.  In section \ref{sec:minimal_residual}, we present a natural minimal  residual based method for the approximation in a nonlinear subset $\Sc_{X}$, and we analyze a simple gradient algorithm in $\Sc_{X}$. We discuss the conditioning issues that restrict the applicability of such algorithms when usual residual norms are used, and the interest of using an ideal measure of the residual. 
In section \ref{sec:perturbation}, we introduce the perturbed ideal minimal residual approach. A gradient-type algorithm is introduced and analyzed and we prove the convergence of this algorithm towards a neighborhood of the best approximation in $\Sc_{X}$. Practical computational aspects are detailed in section \ref{sec:practical_aspects}. In section \ref{sec:greedy}, we analyze a weak greedy algorithm using the perturbed ideal minimal residual method for the computation of greedy corrections. 
In section \ref{sec:example},
a detailed numerical example will illustrate the proposed method. The example is a stochastic reaction-advection-diffusion problem  which is discretized using Galerkin stochastic methods. In particular, this example will illustrate the  possibility to introduce norms that are adapted to some quantities of interest and the ability of the method to provide (quasi-)best low-rank approximations {in that context.}

\section{Functional framework for weakly coercive problems}\label{sec:weakly_coercive_problem}

\subsection{Notations}
For a given Hilbert space $H$, we denote by $\langle\cdot
,\cdot\rangle_H$ the inner product in $H$ and by $\|\cdot\|_H$ the
associated norm. We denote by $H'$ the topological dual of $H$ and
by $\langle\cdot ,\cdot\rangle_{H',H}$ the duality pairing between
$H$ and $H'$. For $v\in H$ and $\varphi\in H'$, we denote
$\varphi(v)=\langle \varphi, v\rangle_{H',H}$.
%If there is no
%ambiguity, we simply denote $\langle\cdot
%,\cdot\rangle_{H',H}=\langle\cdot ,\cdot\rangle$.
We denote by $R_H
: H\rightarrow H'$ the Riesz isomorphism defined by
$$
 \langle v,w\rangle_{H} = \langle v , R_H w\rangle_{H,H'} =
 \langle R_H v , w\rangle_{H',H} =\langle
R_H v , R_H w\rangle_{H'} \quad \forall v,w\in H.
$$

\subsection{Weakly coercive problems}

We denote by $X$ (resp. $Y$) a Hilbert space equipped with inner
product $\langle\cdot ,\cdot\rangle_X$ (resp. $\langle\cdot
,\cdot\rangle_Y$) and associated norm $\|\cdot\|_X$ (resp.
$\|\cdot\|_Y$). Let $a:X\times Y \rightarrow\xR$ be a bilinear form
and let $b\in Y'$ be a continuous linear form on $Y$. We consider
the variational problem: find $u\in X$ such that
\begin{align}
a(u,v)= b(v)\quad \forall v\in Y.\label{eq:problem_var}
\end{align}
%with $b(v)=\langle b ,v\rangle_{Y',Y}$.
We assume that $a$ is continuous and weakly coercive, that means
that there exist constants $\alpha$ and $\beta$ such that
\begin{align}
&\sup_{v\in X}\sup_{w\in Y} \frac{a(v,w)}{\Vert v\Vert_X\Vert
w\Vert_Y} = \beta <+\infty,\label{eq:a_continuous}\\
&\adjustlimits\inf_{v\in X}\sup_{w\in Y} \frac{a(v,w)}{\Vert
v\Vert_X\Vert w\Vert_Y} = \alpha >0,
 \label{eq:bounding_A}\end{align}
 and
\begin{align}
&\sup_{v\in X} \frac{a(v,w)}{\Vert v\Vert_X} >0\quad \forall w\neq 0
\text{ in } Y.
 \label{eq:injective_A}
\end{align}
We introduce the linear continuous operator $A:X\rightarrow Y'$ such
that for all $(v,w)\in X\times Y$,
$$
a(v,w) = \langle Av,w\rangle_{Y',Y}.
$$
We denote by $A^*:Y\rightarrow X'$ the adjoint of $A$, defined by
$$
\langle Av,w\rangle_{Y',Y} = \langle v,A^*w\rangle_{X,X'}\quad
\forall (v,w)\in X\times Y.
$$
Problem \eqref{eq:problem_var} is therefore equivalent to find $u\in
X$ such that
\begin{align}
Au=b. \label{eq:problem}
\end{align}
Properties \eqref{eq:a_continuous},\eqref{eq:bounding_A} and
\eqref{eq:injective_A} imply that $A$ is a norm-isomorphism from $X$
to $Y'$ such that for all $v\in X$,
\begin{align}
\alpha\Vert v\Vert_X \le \Vert Av\Vert_{Y'} \le  \beta \Vert
v\Vert_X \label{eq:Av_bounds}
\end{align}
ensuring the well-posedness of problem
 \eqref{eq:problem}\cite{Ern2004}. {The norms of $A$ and its inverse $A^{-1}$ are such that 
 $\Vert A\Vert_{X\rightarrow Y'}  = \beta$ and $\Vert
A^{-1}\Vert_{Y'\rightarrow X}=\alpha^{-1}$.  
%$$
%\Vert A\Vert_{X\rightarrow Y'}  = \sup_{0\neq v\in X} \frac{\Vert
%Av\Vert_{Y'}}{\Vert v\Vert_X} = \beta ,$$ and
%$$\Vert
%A^{-1}\Vert_{Y'\rightarrow X}= \sup_{0\neq  y\in Y'} \frac{\Vert
%A^{-1} y\Vert_{X}}{\Vert y\Vert_{Y'}} = \left(\inf_{0\neq v\in X}
%\frac{\Vert A v\Vert_{Y'}}{\Vert  v\Vert_X}\right)^{-1} = \alpha^{-1}.
%$$
Then, the condition number of the operator $A$ 
is 
$$\kappa(A) = \Vert
A\Vert_{X\rightarrow Y'} \Vert A^{-1}\Vert_{Y'\rightarrow X} = \frac{\beta}{\alpha} \ge 1.
$$}

\section{Approximation in low-rank tensor subsets}\label{sec:approx_tensors}

\subsection{Hilbert tensor spaces} We here briefly recall basic definitions on
Hilbert tensor spaces (see \cite{Hackbusch2012}). We consider Hilbert spaces $X_\mu$, $1\le \mu\le d$, equipped with
norms $\Vert\cdot\Vert_{X_\mu}$ and associated inner products
$\langle\cdot,\cdot\rangle_\mu$\footnote{e.g. $X_\mu=\xR^{n_\mu}$ equipped
with the Euclidian norm, or $X_\mu = H^{k}_0(\Omega_\mu$), $k\ge 0$, a Sobolev
space of functions defined on a domain  $\Omega_{\mu}$.}. We denote by
$\otimes_{\mu=1}^d v^\mu = v^1\otimes \hdots\otimes v^d$, $v^\mu\in X_\mu$, an
elementary tensor. We then define the algebraic tensor product space as the
linear span of elementary tensors: $${}_a\bigotimes_{\mu=1}^d X_\mu =
\mathrm{span}\{\otimes_{\mu=1}^d v^\mu\; :\; v^\mu\in X_\mu,1\le \mu\le d\}.$$
A Hilbert tensor space $X$ equipped with the norm $\Vert\cdot\Vert_X$ is then
obtained by the completion with respect to $\Vert\cdot\Vert_X$ of the algebraic
tensor space, i.e. $$X= \overline{{}_{a}\bigotimes_{\mu=1}^d X_\mu}^{\Vert
\cdot\Vert_X} = {}_{\Vert \cdot\Vert_X}\bigotimes_{\mu=1}^d X_\mu. $$ Note that
for finite dimensional tensor spaces, the resulting space $X$ is independent of
the choice of norm and coincides with the normed algebraic tensor space.

A natural inner product on $X$ is induced by inner products
$\langle\cdot,\cdot\rangle_\mu$ in $X_\mu$, $1\le\mu\le d$. It is
defined for $v=\otimes_{\mu=1}^d v^\mu$ and $w=\otimes_{\mu=1}^d
w^\mu$ by
$$
\langle v,w\rangle_X = \prod_{\mu=1}^d \langle
v^\mu,w^\mu\rangle_\mu
$$
and extended by linearity on the whole algebraic tensor space. This
inner product is called the \emph{induced (or canonical) inner
product} and the associated norm the \emph{induced (or canonical)
norm.}

\subsection{Classical low-rank tensor subsets}\label{sec:tensor_subsets}
Low-rank tensor subsets $\Sc_X$ of a tensor space
$X={}_{\Vert\cdot\Vert}\bigotimes_{\mu=1}^d X_\mu$ are subsets of
the algebraic tensor space ${}_{a}\bigotimes_{\mu=1}^d X_\mu$, which
means that elements $v\in \Sc_X$ can be written under the form
\begin{align}
v = \sum_{i_1\in I_1}\hdots\sum_{i_d\in I_d} \alpha_{i_1,\hdots,i_d}
\bigotimes_{\mu=1}^d v_{i_\mu}^\mu,\label{eq:tensor_form}
\end{align}
where $\alpha = (\alpha_i)_{i\in I} \in \xR^I$, with $I :=
I_1\times \hdots \times I_d$, is a set of real coefficients that
possibly satisfy some constraints, and $(v_{i_\mu}^\mu)_{i_\mu\in
I_\mu}\in (X_\mu)^{I_\mu}$, for $1\le \mu\le d $, is a set of
vectors that also possibly satisfy some constraints (e.g.
orthogonality).

Basic low-rank tensor subsets are the set of tensors with canonical rank bounded by $r$:
$$
\Rc_r(X) = \left\{v=\sum_{i=1}^r \otimes_{\mu=1}^d
v_i^\mu\;:\;v_i^\mu \in X_\mu\right\},
$$
and the set of Tucker tensors with multilinear rank bounded by $r=(r_1,\hdots,r_d)$:
\begin{align*}
\Tc_r(X)&= \left\{v=\sum_{i_1=1}^{r_1} \hdots \sum_{i_d=1}^{r_d}
\alpha_{i_1,\hdots,i_d} \otimes_{\mu=1}^d
v_{i_\mu}^\mu\;:\;v_{i_\mu}^\mu \in X_\mu,\;
\alpha_{i_1,\hdots,i_d}\in\xR\right\}
%\\&=  \left\{v \in X :
%\begin{array}{l} \text{there exist linear subspaces $U_\mu$ with } \\
%\text{$dim(U_\mu)=r_\mu$,  $1\le \mu\le d$, such that }v\in
%{}_{a}\bigotimes_{\mu=1}^d U_\mu
%  \end{array}
% \right\}.
\end{align*}
Other low-rank tensor subsets have been recently introduced, such as Tensor
{Train tensors \cite{Oseledets2011,Holtz2012manifolds}} or more general tree-based Hierarchical Tucker
tensors \cite{Hackbusch2009,Falco2013}, these tensor subsets corresponding to a form
\eqref{eq:tensor_form} with a particular structure of tensor
$\alpha$. Note that for the case $d=2$, all the above tensor subsets 
coincide.

\begin{rmrk}\label{rem:discretization} From a numerical point of view, the
approximate solution of the variational  problem  \eqref{eq:problem_var}
requires an additional discretization which consists in  introducing an
approximation space $\widetilde X = \otimes_{\mu=1}^d \widetilde  X_\mu$, where
the $\widetilde X_\mu\subset X_\mu$ are finite dimensional approximation spaces
  (e.g. finite element spaces). Then, approximations are searched in  low-rank
tensor subsets $\Sc_X$ of $\widetilde X$  (e.g. $\Rc_r(\widetilde X)$ or
$\Tc_{r}(\widetilde  X))$, thus introducing two levels of discretizations. In
the following, we adopt a general  point of view where $X$ can either denote an
infinite dimensional space, an  approximation space obtained after the
discretization of the variational  problem, or even finite dimensional
Euclidian spaces for problems written in an  algebraic form. 
%From a numerical point of view, the  
%resolution of general partial differential equation as 
%\eqref{eq:intro_problem} in Hilbert space $X$  is usually 
%addressed using a specific discretization (e.g. finite element method). 
%In consequence, it seems natural to exhibit an approximation tensor 
% subset taking into account this last aspect. This should be done,  
% using for example $\widetilde{\Rc}_r(X)$ instead of $\Rc_r(X)$ (or 
% in similar way $\widetilde{\Tc}_r(X)$ for $\Tc_r(X)$) defined as  
%$$
%\widetilde{\Rc}_r(X) =\left\{v=\sum_{i=1}^r \otimes_{\mu=1}^d v_i^\mu\;:\;v_i^\mu \in \widetilde{X}_\mu\right\}, 
%$$ 
%with $\widetilde{X}_\mu \subset X_\mu$ a finite dimensional space.\\ 
%In this way, one should note that $\widetilde{\Rc}_r(X)$ is the result of  
%two levels of discretization; the first one associated to classical linear 
%approximation and the second coming from non linear approximation 
%using separated representations in tensor subsets. 
\end{rmrk}

\subsection{Best approximation in tensor subsets}
Low-rank tensor approximation methods consist in computing an approximation
of a tensor $u\in X$ in a suitable low-rank subset $\Sc_X$ of $X$.
The best approximation of $u$ in $\Sc_X$ is defined by
\begin{align}
\min_{v\in \Sc_X} \Vert u-v\Vert_X. \label{eq:best_approximation_SX}
\end{align}
{The previously mentioned classical tensor subsets are neither linear subspaces nor convex sets.}
However, they usually satisfy properties that give sense to the
above best approximation problem. We consider the case that $\Sc_X$
satisfies the following properties:
\begin{align}
&\text{$\Sc_X$ is weakly closed (or simply closed in finite dimension),} %\tag{A1}
\label{SX_weakly_closed} \\
&\text{$\Sc_X\subset \gamma \Sc_X$ for all
$\gamma\in\xR$.}\label{SX_cone} %\tag{A2}
\end{align}
Property \eqref{SX_cone} is satisfied by all the classical tensor
subsets mentioned above (canonical tensors, Tucker and tree-based Hierarchical
Tucker tensors). Property \eqref{SX_weakly_closed} ensures the
existence of solutions to the best approximation problem
\eqref{eq:best_approximation_SX}. This property, under some suitable
conditions on the norm $\Vert\cdot\Vert_X$ (which is naturally
satisfied in finite dimension), is verified by most tensor subsets
used for approximation (e.g. the set of  tensors with bounded canonical rank
for $d=2$, the set of elementary tensors $\Rc_1$ for $d\ge 2$
\cite{Falco2011}, the sets of Tucker or tree-based Hierarchical Tucker tensors
\cite{Falco2012-minimal}).\\
\par

We then introduce the set-valued map $\Pi_{\Sc_X}:X\rightarrow
2^{\Sc_X}$ that
 associates to an element $u\in X$ the set of best approximations
 of $u$ in $\Sc_X$:
\begin{align}
\Pi_{\Sc_X}(u) = \arg\min_{v\in\Sc_X}\Vert u-v\Vert_X.\label{eq:defPiSX}
\end{align}
Note that if $\Sc_X$ were a closed linear subspace or a closed
convex set of $X$, then $\Pi_{\Sc_X}(u)$ would be a singleton and
$\Pi_{\Sc_X}$ would coincide with the classical definition of the
metric projection on $\Sc_X$. Property \eqref{SX_cone} still implies
the following property of projections: for all $v\in X$ and for all
$w\in \Pi_{\Sc_X}(v)$,
\begin{align}
\Vert v-w \Vert_X^2 = \Vert v \Vert_X^2 - \Vert w\Vert_X^2 \quad
 \text{with} \quad \Vert w \Vert_X = \sigma(v;\Sc_X) =
\max_{z\in \Sc_X} \frac{\langle v,z\rangle_{X}}{\Vert z\Vert_X} .
\end{align}
$\Pi_{\Sc_X}(v)$ is therefore a subset of the sphere of radius
$\sigma(v;\Sc_X)$ in $X$. In the following, we will use the
following abuse of notation: for a subset $S\subset X$ and for $w\in
X$, we define
$$
\Vert S - w\Vert_X := \sup_{v\in S} \Vert v-w\Vert_X
$$
With this convention, we have $\Vert \Pi_{\Sc_X}(v)\Vert_X =
\sigma(v;\Sc_X)$ and \begin{align}\Vert \Pi_{\Sc_X}(v) - v\Vert_X^2
= \Vert v\Vert_X^2 - \Vert \Pi_{\Sc_X}(v)\Vert_X^2.
\end{align}

\section{Minimal residual based approximation}\label{sec:minimal_residual}

We now consider that problem \eqref{eq:problem} is formulated in
tensor Hilbert spaces $X={}_{\Vert \cdot\Vert_X}\bigotimes_{\mu=1}^d
X_\mu$ and $Y={}_{\Vert \cdot\Vert_Y}\bigotimes_{\mu=1}^d Y_\mu.$
%We
%typically admit that operator $A$ and right-hand side admit a
%low-rank tensor representation of the form
%$$A =
%\sum_{i=1}^{r_A} \otimes_{\mu=1}^d A_i^\mu \quad \text{and} \quad b
%= \sum_{i=1}^{r_b} \otimes_{\mu=1}^d b_i^\mu.
%$$
The aim is here to find an approximation of the solution $u$ of
problem \eqref{eq:problem} in a given tensor subset $\Sc_X\subset
X$.

\subsection{Best approximation with respect to residual norms}
Since the solution $u$ of problem \eqref{eq:problem} is not
available, the best approximation problem
\eqref{eq:best_approximation_SX} cannot be solved directly.
However, tensor approximations can be defined using the residual of
equation \eqref{eq:problem}, which is a computable information.  An
approximation of $u$ in $\Sc_X$ is then defined by the minimization
of a residual norm:
\begin{align}
\min_{v\in\Sc_X} \Vert Av-b\Vert_{Y'} =\min_{v\in\Sc_X} \Vert
A(v-u)\Vert_{Y'} .\label{eq:minimization_SX_residual}
 \end{align}
Assuming that we can define a tangent space $T_v (\Sc_X)$ to $\Sc_X$
at $v\in\Sc_X$, the stationarity condition of functional $J:
v\mapsto \Vert A(v-u)\Vert_{Y'}^2$ at $v\in\Sc_X$ is 
$$\langle J'(v),\delta v\rangle_{X',X} = 0\quad \forall \delta v\in
T_v(\Sc_X),
$$
or equivalently, noting that the gradient of $J$ at $v$ is $J'(v) =
A^*R_Y^{-1}(Av-b)\in X'$,
$$\langle Av-b,A\delta v\rangle_{Y'} =
0\quad \forall \delta v\in T_v(\Sc_X).
$$

\subsection{Ideal choice of the residual norm}\label{sec:ideal_norms}
When approximating $u$ in $\Sc_X$ using
\eqref{eq:minimization_SX_residual}, the obtained approximation
depends on the choice of the residual norm.
If we want to find a best approximation of $u$ with respect to the norm $\Vert\cdot\Vert_X$,
then the residual norm should be chosen \cite{Cohen2012,Dahmen2012} such that
$$
\Vert A(v-u)\Vert_{Y'} = \Vert v-u\Vert_X \quad \forall v\in X,
$$
or equivalently such that the following relation between inner
products holds:
\begin{align}
\langle v,w\rangle_X = \langle Av,Aw\rangle_{Y'}\quad  \forall
v,w\in X. \label{eq:inner_product_relation_1}
\end{align}
This implies
\begin{align*} \langle v,w\rangle_X &= \langle
Av,R_Y^{-1}Aw\rangle_{Y',Y}
%\\
%&= \langle v,A^*
%R_Y^{-1}Aw\rangle_{X,X'} \\&
=\langle v,R_X^{-1}A^* R_Y^{-1}Aw\rangle_{X},
\end{align*}
for all $v,w\in X$, and therefore, by identification,
\begin{align}
I_X = R_X^{-1}A^* R_Y^{-1}A  \Leftrightarrow R_Y = A R_X^{-1}A^*
\Leftrightarrow R_X = A^* R_Y^{-1}A \label{eq:relation_RX_RY} .
\end{align}
Also, since
\begin{align*}
\langle v,w\rangle_Y &= \langle R_Y v,w\rangle_{Y',Y} =
\langle A R_X^{-1}A^* v,w\rangle_{Y',Y} \\
&= \langle R_X^{-1}A^* v, A^* w\rangle_{X,X'} = \langle A^* v,
A^*w\rangle_{X'}
\end{align*}
for all $v,w\in Y$, we also have that
\eqref{eq:inner_product_relation_1} is equivalent to the following
relation:
\begin{align}
\langle v,w\rangle_{Y} = \langle A^*v,A^*w\rangle_{X'}\quad \forall
v,w\in Y. \label{eq:inner_product_relation_2}
\end{align}
Note that \eqref{eq:inner_product_relation_1} and
\eqref{eq:inner_product_relation_2} respectively impose
\begin{align}
\Vert v\Vert _X &=\Vert Av\Vert _{Y'} \text{ and } \Vert w\Vert _Y = \Vert A^*w\Vert _{X'}.
 \label{eq:norms}
\end{align}
This choice implies that the weak coercivity and continuity constants are such that $\alpha=\beta=1$, and therefore
$$\kappa(A)=1,$$ meaning that problem \eqref{eq:problem} is
ideally conditioned.
\\\par
 In
practice, we will first define the inner product $\langle
\cdot,\cdot\rangle_X$ and the other inner product $\langle
\cdot,\cdot\rangle_Y$ will be deduced from
\eqref{eq:inner_product_relation_2}.

\begin{xmpl}
Consider that $X=Y$ and let $A=B+C$ with $B$ a symmetric coercive
and continuous operator and $C$ a skew-symmetric operator. We equip
$X$ with inner product $\langle v,w\rangle_X = \langle B
v,w\rangle_{X',X}$, which corresponds to $R_X=B$. Therefore,
$$
\Vert v\Vert_Y^2 = \Vert A^*v \Vert_{X'}^2 = \Vert B v \Vert_{X'}^2
 + \Vert C v \Vert_{X'}^2 = \Vert v \Vert_{X}^2
 + \Vert C v \Vert_{X'}^2.
$$
$\Vert v\Vert_Y $ corresponds to the graph norm of the
skew-symmetric part $C$ of the operator $A$. When $C=0$, we simply have
$\Vert v\Vert_Y^2 = \Vert v\Vert_X^2$.
\end{xmpl}
%\begin{xmpl}[Advection-diffusion problem]
%Consider an advection diffusion problem with
%\begin{align}
%Au = -\Delta u + c\cdot \nabla u,
%\end{align}
%and with $X=Y = H^1_0(\Omega)$. $X'=Y'=H^{-1}(\Omega)$.
%%We denote by $H=\xLtwo(\Omega)$,
%%such that we have
%%have two Gelfand triplet $X\hookrightarrow H \hookrightarrow X'$
%%and  $Y\hookrightarrow H \hookrightarrow Y'$.
%We denote by $A_s = -\Delta$ (resp. $A_{sk}=c\cdot \nabla u$) the
%symmetric (resp. skew-symmetric) part of $A$ and we equip $Y$ with
%inner product $\langle v,w\rangle_Y = \langle A_s
%v,w\rangle_{Y',Y}$, where duality pairing $\langle
%\cdot,\cdot\rangle_{Y',Y}$ is induced by the standard $\xLtwo$ inner
%product (using Gelfand triple $H^1_0(\Omega)\hookrightarrow
%\xLtwo(\Omega)\hookrightarrow H^{-1}(\Omega)$). $\langle
%\cdot,\cdot\rangle_Y$ corresponds to the classical inner product on
%$H_0^1(\Omega)$, with $R_Y=A_s$. Inner product on $X$ is then
%defined by relation \eqref{eq:inner_product_relation_1}, which is
%equivalent to
%$$
%\Vert v\Vert_X^2 = \Vert v \Vert_Y^2 + \Vert A_{sk} v \Vert_{Y'}^2.
%$$
%\end{xmpl}
\begin{xmpl}[Finite dimensional problem]
Consider the case of finite dimensional tensor spaces $X = Y =
\xR^{n_1\times \hdots \times n_d}$, e.g. after a discretization
step for the solution of a high-dimensional partial differential equation. The duality pairings are induced by
the standard canonical inner product. We can choose for $\langle
\cdot,\cdot \rangle_X $ the canonical inner product on
$\xR^{n_1\times \hdots \times n_d}$, which corresponds to
$R_X=I_X$, the identity on $X$. Then, inner product on $Y$ is
defined by relation \eqref{eq:inner_product_relation_2}, which
implies
$$
\langle v,w\rangle_{Y} = \langle A^*v, A^*w\rangle_{X}
\quad\text{and} \quad R_Y = AA^*.
$$
\label{ex:symmetric}
\end{xmpl}

\subsection{Gradient-type algorithm}\label{sec:gradient_algorithm}

For solving \eqref{eq:minimization_SX_residual}, we consider the following basic gradient-type algorithm: letting $u^0=0$, we construct a sequence
$\{u^k\}_{k\ge 0}$ in $\Sc_X$ and a sequence $\{y^k\}_{k\ge 0}$ in
$Y$ defined for $k\ge 0$ by
 \begin{equation}\label{eq:gradient_algorithm}
 \left\{ \begin{aligned}
  &y^k = R_Y^{-1}(Au^k-b)\\
  &u^{k+1} \in \Pi_{\Sc_X}(u^k - \rho R_X^{-1} A^* y^k)
  \end{aligned}\right.
\end{equation}
 with $\rho>0$. Equations \eqref{eq:gradient_algorithm} yield
 \begin{align*}
  u^{k+1} \in \Pi_{\Sc_X}(u + B_\rho(u^k-u)),
 \end{align*}
 with $B_\rho=I_X-\rho R_X^{-1} A^*R_Y^{-1}A$ a symmetric operator from $X$ to $X$.
 For all $v\in X$, $$\frac{\langle B_\rho v,v \rangle_X}{\Vert
v\Vert_X^2} = 1-\rho \frac{\Vert Av\Vert_{Y'}^2}{\Vert
v\Vert_X^2}.$$
{Here, we assume that $\Vert\cdot\Vert_X$ and $\Vert\cdot\Vert_Y$ do not necessarily satisfy the relation \eqref{eq:norms} (i.e. $\frac{\alpha}{\beta}\neq 1$)}.
From \eqref{eq:Av_bounds}, we deduce that the
eigenvalues of $B_\rho$ are in the interval
$[1-\rho\beta^2,1-\rho\alpha^2]$. The spectral radius of $B_\rho$ is
therefore bounded by $$\gamma(\rho) = \max\{\vert
1-\rho\beta^2\vert,\vert1-\rho\alpha^2\vert\}.$$

\begin{prpstn}\label{prop:gradient_algorithm_convergence}
Assuming $\gamma(\rho)<1/2$, the sequence $\{u^k\}_{k\ge 1}$ defined
by \eqref{eq:gradient_algorithm} is such that
\begin{align}
\Vert u^{k} - u\Vert_X &\le (2\gamma)^{k} \Vert u^{0} - u\Vert_X +
\frac{1}{1-2\gamma}\Vert u -
\Pi_{\Sc_X}(u)\Vert_X\label{eq:gradient_algorithm_convergence_iterates}
\end{align}
and
\begin{align}
\lim\sup_{k\to\infty}\Vert u^k - u\Vert_X \le \frac{1}{1-2\gamma}
\Vert u - \Pi_{\Sc_X}(u)\Vert_X\label{eq:gradient_algorithm_limsup}
\end{align}
%where $\gamma = \max\{\vert
%1-\rho\beta^2\vert,\vert1-\rho\alpha^2\vert\}$.
\end{prpstn}
\begin{proof}
Denoting $v^k =u^k-u$, we have
\begin{align*}
\Vert u^{k+1} - u\Vert_X &\le \Vert \Pi_{\Sc_X}(u + B_\rho v^k)- u\Vert_X\\
&\le\Vert \Pi_{\Sc_X}(u + B_\rho v^k)- (u + B_\rho v^k)\Vert_X +
\Vert  B_\rho v^k\Vert_X \\
&\le \Vert w-(u + B_\rho v^k)\Vert_X+ \Vert  B_\rho v^k\Vert_X
\end{align*}
for all $w\in\Sc_X$. In particular, this inequality is true for all
$w\in\Pi_{\Sc_X}(u)$, and therefore, taking the supremum over all
$w\in \Pi_{\Sc_X}(u)$, we obtain
\begin{align*}
\Vert u^{k+1} - u\Vert_X &\le \Vert \Pi_{\Sc_X}(u)-(u + B_\rho
v^k)\Vert_X+ \Vert  B_\rho v^k\Vert_X
\\
&\le\Vert \Pi_{\Sc_X}(u)-u \Vert_X+ 2 \Vert  B_\rho v^k\Vert_X
\end{align*}
Since $\Vert B_\rho v\Vert_X\le \gamma \Vert v\Vert_X$ for all $v\in
X$ and since $2\gamma<1$, we have
\begin{align*}
\Vert u^{k+1} - u\Vert_X &\le\Vert \Pi_{\Sc_X}(u)-u \Vert_X+ 2
\gamma \Vert u-u^k\Vert_X\\
&\le  (2\gamma)^{k+1} \Vert u^{0} - u\Vert_X +
\frac{1-(2\gamma)^{k+1}}{1-2\gamma}\Vert u - \Pi_{\Sc_X}(u)\Vert_X
\end{align*}
from which we deduce
\eqref{eq:gradient_algorithm_convergence_iterates} and
\eqref{eq:gradient_algorithm_limsup}.
\end{proof}

The condition $\gamma(\rho)<1/2$ imposes $\frac{\beta}{\alpha}<
\sqrt{3}$ and $\rho\in(\frac{1}{2\alpha^2},\frac{3}{2\beta^2})$. The
condition $\frac{\beta}{\alpha}< \sqrt{3}$ is a very restrictive
condition which is in general not satisfied without an excellent
preconditioning of the operator $A$.

However, with the ideal choice of norms introduced in the previous
section (equation \eqref{eq:norms}), we have $\alpha=\beta=1$ and
$B_\rho=(1-\rho)I_X$. That means that the problem is ideally
conditioned and we have convergence for all
$\rho\in[\frac{1}{2},\frac{3}{2}]$ towards a neighborhood of
$\Pi_{\Sc_X}(u)$ of size $\frac{2\gamma}{1-2\gamma}\Vert u -
\Pi_{\Sc_X}(u)\Vert_X$ with $\gamma = \vert 1-\rho\vert$.
\begin{crllr}
Assume that \eqref{eq:norms} is satisfied. Then, if $\rho
\in[\frac{1}{2},\frac{3}{2}]$, the sequence $\{u^k\}_{k\ge 1}$
defined by \eqref{eq:gradient_algorithm} verifies
\eqref{eq:gradient_algorithm_convergence_iterates} and
\eqref{eq:gradient_algorithm_limsup} with $\gamma(\rho)=\vert
1-\rho\vert$. Moreover, if $\rho=1$, then $u^1\in\Pi_{\Sc_X}(u)$,
which means that the algorithm converges in one iteration {for any}
 initialization $u^0$.
\end{crllr}

%Moreover, choosing $\rho=1$, we have $B_\rho=0$ and therefore
%convergence to an element of $\Pi_{\Sc_X}(u)$ in only 1 iteration
%since
%$$
%\Vert u^1 - u\Vert_X = \Vert \Pi_{\Sc_X}(u) - u \Vert_X
%$$
%implies that $u^1\in\Pi_{\Sc_X}(u)$.

\section{Perturbation of the ideal
approximation}\label{sec:perturbation}

We now consider that function spaces $X$ and $Y$ are equipped with
norms satisfying the ideal condition
\begin{align}
\Vert Av\Vert_{Y'} = \Vert v\Vert_X\quad\forall v\in
X.\label{eq:norms_constraint}
\end{align}
The solution of problem \eqref{eq:minimization_SX_residual} using
this ideal choice of norms is therefore equivalent to the best
approximation problem \eqref{eq:best_approximation_SX}, i.e.
\begin{align}
\min_{v\in\Sc_X} \Vert Av-b\Vert_{Y'} = \min_{v\in\Sc_X} \Vert
v-u\Vert_X. \label{eq:min_res_equivalent}
\end{align}
Unfortunately, the computation of the solution of
\eqref{eq:min_res_equivalent} would require the solution of the
initial problem. We here propose to introduce a computable
perturbation of this ideal approach.

\subsection{Approximation of the ideal approach}
Following the idea of \cite{Cohen2012}, the problem \eqref{eq:min_res_equivalent} is replaced by the following problem:
\begin{align}
\min_{v\in\Sc_X} \Vert \Lambda^\delta(R_Y^{-1}(Av-b))\Vert_{Y},
\label{eq:min_res_approximation}
\end{align}
where $\Lambda^\delta:Y\rightarrow Y$ is a mapping that provides an approximation $\Lambda^\delta(r)$ of 
the residual 
 $r= R_Y^{-1}(Av-b)\in Y$ with a controlled
relative precision $\delta>0$, i.e.
%\begin{align*}
$\Vert \Lambda^\delta(r) -r\Vert_Y \le \delta \Vert
r\Vert_Y.%\label{eq:delta_prox_Lambda}
$ %\end{align*} 
We will then assume
that the mapping $\Lambda^\delta$ is such that:
\begin{align}
\Vert \Lambda^\delta(y) -y\Vert_Y \le \delta \Vert y\Vert_Y, \quad
\forall y\in \Dc_Y = \left\{R_Y^{-1}(Av-b);v\in\Sc_X
\right\}.\label{eq:delta_prox_Lambda}
\end{align}
As we will see in the following algorithm, it is sufficient for
$\Lambda^\delta$ to well approximate residuals that are in the
subset $\Dc_Y$ whose content depends
on the chosen subset $\Sc_X$ and on the operator and right-hand side
of the problem.

%{\begin{rmrk}
%A practical construction of an approximation of $\Lambda^\delta(r)$ will be
%detailed in Section \ref{sec:practical_aspects}. % This approximation will not satisfy exactly \eqref{eq:delta_prox_Lambda}.
%\end{rmrk}}
%{A practical construction of an approximation of $\Lambda^\delta(r)$ will be
%detailed in Section \ref{sec:practical_aspects}.} 

\subsection{Quasi-optimal approximations in $\Sc_X$}
 Here we consider the case where we are not able to solve
the best approximation problem in $\Sc_X$ exactly, because there is
no available algorithm for computing a global optimum, or because
the algorithm has been stopped at a finite precision (see section
\ref{sec:best_approx_comput} for practical comments). We introduce a
set of quasi-optimal approximations $\Pi_{\Sc_X}^\eta(u)\subset
\Sc_X$ such that
\begin{align}
\Vert u- \Pi_{\Sc_X}^\eta(u)\Vert_X\le \eta \Vert u-
\Pi_{\Sc_X}(u)\Vert_X\quad  (\eta\ge 1).\label{eq:PiSX_eta}%\inf_{w\in\Sc_X} \Vert u- w\Vert_X.
\end{align}

\begin{rmrk}
Note that by introducing this new perturbation, we are able to
remove the assumption that $\Sc_X$ is closed and to handle the case
where the problem \eqref{eq:min_res_equivalent}
does not have a
solution, i.e. $\Pi_{\Sc_X}(u)=\emptyset$. In this case, we have to
replace $\Vert u- \Pi_{\Sc_X}(u)\Vert_X$ by $\inf_{w\in\Sc_X} \Vert
u- w\Vert_X$ in equation \eqref{eq:PiSX_eta}.
\end{rmrk}

\begin{rmrk}\label{rem:discretization_quasioptimal}
Note that if $\Sc_X$ denotes a low-rank subset of an infinite dimensional space $X$, 
additional approximations have to be introduced from a numerical point of view 
(see remark \ref{rem:discretization}). These additional approximations could be 
also
considered as a perturbation leading to quasi-optimal approximations, where $\eta$ 
takes into account the approximation errors. In numerical examples, we will not adopt this 
point of view and we will consider $X$ as the approximation space and the approximate solution 
in
$X$ of the variational problem will serve as a reference solution.
\end{rmrk}

%%Algorithm \eqref{eq:gradient_algorithm_perturbed} can be replaced by
% \begin{equation}
% \label{eq:gradient_algorithm_perturbed_epsilon}
%  u^{k+1} \in \Pi^\eta_{\Sc_X}(u^k - R_X^{-1} A^*
%  \Lambda^\delta(R_Y^{-1}(Au^k-b))).
%\end{equation}

%
%
%\begin{prpstn}\label{prop:gradient_algorithm_perturbed_epsilon_convergence}
%Assume \eqref{eq:norms_constraint} and \eqref{eq:delta_prox_Lambda},
%with $\delta<\frac{1}{2+\eta}$. Then, the sequence
%$\{u^k\}_{k\ge 1}$ defined by
%\eqref{eq:gradient_algorithm_perturbed_epsilon} is such that
%\begin{align}
% \Vert u^{k} - u\Vert_X &\le
%(2\delta)^{k} \Vert u^{0} - u\Vert_X + \frac{1}{1-2\delta}\Vert u -
%\Pi_{\Sc_X}(u)\Vert_X\label{eq:gradient_algorithm_perturbed_convergence_iterates}
%\end{align}
%and
%\begin{align}
%\lim\sup_{k\to\infty}\Vert u^k - u\Vert_X \le \frac{1}{1-2\delta}
%\Vert u -
%\Pi_{\Sc_X}(u)\Vert_X\label{eq:gradient_algorithm_perturbed_limsup}
%\end{align}
%%where $\gamma = \max\{\vert
%%1-\rho\beta^2\vert,\vert1-\rho\alpha^2\vert\}$.
%\end{prpstn}

\subsection{Perturbed gradient-type algorithm}\label{sec:Perturbed_gradient_algorithm}
For solving \eqref{eq:min_res_approximation}, we now introduce an
algorithm which can be seen as a perturbation of
  the ideal gradient-type
algorithm \eqref{eq:gradient_algorithm} introduced in section
\ref{sec:gradient_algorithm}. Letting $u^0=0$, we construct a
sequence $\{u^k\}_{k\ge 0}\subset \Sc_X$ and a sequence
$\{y^k\}_{k\ge 0}\subset Y$ defined for $k\ge 0$ by
 \begin{equation}\label{eq:gradient_algorithm_perturbed}
 \left\{ \begin{aligned}
  &y^k = \Lambda^\delta(R_Y^{-1}(Au^k-b))\\
  &u^{k+1} \in \Pi^\eta_{\Sc_X}(u^k - R_X^{-1} A^* y^k)
  \end{aligned}\right.
\end{equation}

\begin{prpstn}\label{prop:gradient_algorithm_perturbed_convergence}
Assume \eqref{eq:norms_constraint}, \eqref{eq:delta_prox_Lambda}, and
\eqref{eq:PiSX_eta}, with $\delta(1+\eta)<1$. Then, the
sequence $\{u^k\}_{k\ge 1}$ defined by
\eqref{eq:gradient_algorithm_perturbed} is such that
\begin{align}
 \Vert u^{k} - u\Vert_X &\le
((1+\eta)\delta)^{k} \Vert u^{0} - u\Vert_X +
\frac{\eta}{1-\delta(1+\eta)}\Vert u -
\Pi_{\Sc_X}(u)\Vert_X.
\label{eq:gradient_algorithm_perturbed_convergence_iterates}
\end{align}
%where $\gamma = \max\{\vert
%1-\rho\beta^2\vert,\vert1-\rho\alpha^2\vert\}$.
%We have
%\begin{equation}
%\lim\sup_{k\to\infty}\Vert u^k - u\Vert_X \le
%\frac{\eta}{1-\delta(1+\eta)} \Vert u -
%\Pi_{\Sc_X}(u)\Vert_X,
%\label{eq:gradient_algorithm_perturbed_limsup}
%\end{equation}
\end{prpstn}

\begin{proof}
Equation
 \eqref{eq:gradient_algorithm_perturbed} can also be written
\begin{align*}
  u^{k+1} \in \Pi_{\Sc_X}^\eta(u + B^\delta(u^k-u))
 \end{align*}
 with $B^\delta(v)=v-R_X^{-1} A^*\Lambda^\delta( R_Y^{-1}A(v))$.
Denoting $v^k=u^k-u$, and following the proof of {Proposition}
\ref{prop:gradient_algorithm_convergence}, we obtain
\begin{align*}
\Vert u^{k+1} - u\Vert_X &\le\Vert \Pi^\eta_{\Sc_X}(u + B^\delta
v^k)- (u + B^\delta v^k)\Vert_X +
\Vert  B^\delta v^k\Vert_X \\
&\le \eta\Vert \Pi_{\Sc_X}(u)-(u + B^\delta v^k)\Vert_X+ \Vert
B^\delta v^k\Vert_X\\
&\le \eta\Vert \Pi_{\Sc_X}(u)-u \Vert_X+ (1+\eta) \Vert B^\delta
v^k\Vert_X
\end{align*}
Moreover, using \eqref{eq:norms_constraint} and
\eqref{eq:relation_RX_RY}, we have
\begin{align*}
\Vert B^\delta v^k\Vert_X &= \Vert v^k- R_X^{-1}
A^*\Lambda^\delta (R_Y^{-1}Av^k) \Vert_X\\
&=\Vert A v^k- A R_X^{-1} A^*\Lambda^\delta (R_Y^{-1}Av^k
)\Vert_{Y'}\\
&=\Vert R_Y^{-1} A v^k- \Lambda^\delta (R_Y^{-1}Av^k )\Vert_{Y}.
\end{align*}
Noting that $R_Y^{-1} A v^k=R_Y^{-1} (A u^k - b)$ belongs to the subset
$\Dc_Y$, we deduce from assumption \eqref{eq:delta_prox_Lambda} and
equation \eqref{eq:norms_constraint} that
\begin{align*}
\Vert B^\delta v^k\Vert_X &\le \delta \Vert R_Y^{-1} A v^k\Vert_{Y}
= \delta \Vert v^k\Vert_X.
\end{align*}
Denoting $\delta_\eta=\delta(1+\eta)<1$, we finally have
\begin{align*}
\Vert u^{k+1} - u\Vert_X &\le  \eta\Vert \Pi_{\Sc_X}(u)-u
\Vert_X+ \delta_\eta \Vert u^k-u\Vert_X\\
&\le \delta_\eta^{k+1} \Vert u^{0} - u\Vert_X +
\eta\frac{1-\delta_\eta^{k+1}}{1-\delta_\eta}\Vert u -
\Pi_{\Sc_X}(u)\Vert_X,
\end{align*}
from which we deduce
\eqref{eq:gradient_algorithm_perturbed_convergence_iterates}.
%and \eqref{eq:gradient_algorithm_perturbed_limsup}  
\end{proof}

\paragraph{\it Comments}
We note the sequence converges towards a neighborhood of
$\Pi_{\Sc_X}(u)$ whose size is
$\frac{\eta-1+(1+\eta)\delta}{1-(1+\eta)\delta}\Vert u -
\Pi_{\Sc_X}(u)\Vert_X$. Indeed,
{\eqref{eq:gradient_algorithm_perturbed_convergence_iterates}} implies
that
\begin{align}
\Vert u - \Pi_{\Sc_X}(u)\Vert_X\le \Vert u - u^k\Vert_X \le
(1+\gamma_k)\Vert u -
\Pi_{\Sc_X}(u)\Vert_X,\label{eq:sequence_uk_tends_to_neighborhood}
\end{align}
with $\limsup_{k\to\infty}\gamma_k\le
\frac{\eta-1+(1+\eta)\delta}{1-(1+\eta)\delta}$. Therefore, the
sequence tends to provide a good approximation of the best
approximation of $u$ in $\Sc_X$, and the parameters $\delta$ and
$\eta$
 control the quality of this approximation. Moreover, equation
\eqref{eq:gradient_algorithm_perturbed_convergence_iterates}
indicates that the sequence converges  quite rapidly to this
neighborhood. Indeed, in the first iterations, when the error $\Vert
u-u^k\Vert_X$ is dominated by the first term $((1+\eta)\delta)^k
\Vert u-u^0\Vert_X$, the algorithm has at least a linear convergence with convergence rate 
$(1+\eta)\delta$ (note that for $\eta\approx 1$,
the convergence rate is very high for small $\delta$). Once both
error terms are balanced, the error stagnates at the value
$\frac{\eta}{1-(1+\eta)\delta}\Vert u - \Pi_{\Sc_X}(u)\Vert_X$. Note
that when $\delta\to 0$, we recover an ideal algorithm  with a
convergence in only one iteration to an element of the set 
$\Pi_{\Sc_X}^\eta(u)$ of quasi-best approximations of $u$ in $\Sc_{X}$. 
\begin{rmrk}
Even if $\Sc_X$ is chosen as a subset of low-rank tensors, the
subset $\Dc_Y$ defined in \eqref{eq:delta_prox_Lambda} possibly contains tensors with high ranks (or even
tensors with full rank) that are not easy to approximate with a
small precision $\delta$ using low-rank tensor representations.
However, the algorithm only requires
 to well approximate the sequence of residuals
$\{R_Y^{-1}(Au^k-b)\}_{k\ge 0}\subset \Dc_Y$, which may be
achievable in practical applications.
\end{rmrk}
%{\begin{rmrk}
%As pointed out the convergence of the gradient type algorithm is insured if $\delta(1+\eta)<1$ meaning that if $\eta$ remains small
%then $\delta$ takes reasonable values. Typically, for algorithms that provides quasi-optimal approximations i.e. with a precision $\eta$ close 
%to 1 then $\delta \in (0,\frac{1}{2})$.
%\end{rmrk}
%}
\subsection{Error indicator}\label{sec:error_indicator}
Along the iterations of algorithm
{\eqref{eq:gradient_algorithm_perturbed}}, an estimation of the true
error $\Vert u-u^k\Vert_X$ can be simply obtained by evaluating the
norm $\Vert y^k\Vert_Y$ of the iterate $y^k =
\Lambda^\delta(r^k)$ with $r^k = R_Y^{-1}(Au^k-b)$. Indeed, from property
\eqref{eq:delta_prox_Lambda}, we have
\begin{align}
(1-\delta)\Vert y\Vert_Y  \le \Vert \Lambda^\delta(y)\Vert_Y \le
(1+\delta)\Vert y\Vert_Y,
\end{align}
for all $y\in\Dc_Y$. Therefore, noting that $r^k \in \Dc_Y$ and
$\Vert r^k \Vert_{Y} =\Vert A(u-u^k)\Vert_{Y'} = \Vert
u-u^k\Vert_X$, we obtain
\begin{align}
(1-\delta)\Vert u-u^k\Vert_X \le \Vert y^k \Vert_Y \le
(1+\delta)\Vert u-u^k\Vert_X.
\end{align}
In other words,
\begin{equation}
 \epsilon^k = \frac{1}{1-\delta}\Vert y^k \Vert_Y
 \label{eq:error_estimator}
\end{equation}
provides an error indicator of the true error $\Vert u-u^k\Vert_X$
with an effectivity index  %
%\begin{equation} 
$\tau^k =\frac{\epsilon^k}{\Vert u-u^k\Vert_X}\in (1,\frac{1+\delta}{1-\delta})$, 
%\label{eq:error_estimator_effectivity}
%\end{equation}
  which is very good for
small $\delta$.

Moreover, if $\Lambda^\delta$ is an orthogonal projection onto some subspace
$Y^\delta\subset Y$, we easily obtain the following improved lower
and upper bounds:
\begin{align}
\sqrt{1-\delta^2}\Vert u-u^k\Vert_X \le \Vert y^k \Vert_Y \le \Vert
u-u^k\Vert_X, \label{eq:error_indicator_bounds}
\end{align}
that means that the following improved error estimator can be
chosen:
\begin{equation}
\hat \epsilon^k = \frac{1}{\sqrt{1-\delta^2}}\Vert y^k \Vert_Y,\label{eq:error_estimator_projection}
\end{equation}
with effectivity index %
%\begin{equation}
%
$\hat \tau^k =\frac{\hat \epsilon^k}{\Vert u-u^k\Vert_X}\in
(1,\frac{1}{\sqrt{1-\delta^2}})$.
%\label{eq:error_estimator_effectivity_projection}.
%\end{equation}

\section{Computational aspects}\label{sec:practical_aspects}

\subsection{Best approximation in tensor
subsets}\label{sec:best_approx_comput}

We here discuss the available algorithms for computing an element in
$\Pi_{\Sc_X}(v)$, that means for solving
\begin{align}
\min_{w\in\Sc_X} \Vert v-w\Vert_X,\label{eq:best_approx_SX_comput}
\end{align}
where $v$ is a given tensor in the tensor space
$X={}_{\Vert\cdot\Vert_X} \bigotimes_{\mu=1}^d X_\mu$ equipped with
norm $\Vert\cdot\Vert_X$, and where $\Sc_X$ is a given tensor
subset. Note that except for the case where $d=2$ and
$\Vert\cdot\Vert_X$ is the induced (canonical) norm, the computation
of a global optimum is still an open problem.

\paragraph{\emph{Canonical norm, $d=2$.}}

For the case $d=2$, we first note that all classical low-rank tensor
formats coincide with the canonical format, that means $\Sc_X =
\Rc_{m}(X)$ for some rank $m$. When the norm $\Vert \cdot\Vert_X$ is
the canonical norm, then $u_m \in \Pi_{\Sc_X}(u)$ coincides with a
rank-$m$ singular value decomposition (SVD) of $u$ (which is possibly not
unique in the case of multiple singular values). Moreover,
$\sigma(u;\Sc_X)^2=\Vert \Pi_{\Sc_X}(u)\Vert_X^2$ is the sum of the
squares of the $m$ dominant singular values of $u$ (see e.g.
\cite{Falco2011}). Efficient algorithms for computing the SVD can
therefore be applied to compute an element in $\Pi_{\Sc_X}(v)$ (a
best approximation). That means that the algorithm
\eqref{eq:gradient_algorithm_perturbed} can be applied with
$\eta=1$.

\paragraph{\emph{Canonical norm, $d>2$.}}
For $d>2$ and when the norm $\Vert \cdot\Vert_X$ is the canonical
norm, different algorithms based on optimization methods have been
proposed for the different tensor formats  {(see e.g. \cite{Espig2012,Holtz2012ALS} or
\cite{Hackbusch2012} for a recent review)}. Very efficient algorithms
based on higher order SVD have also been proposed in \cite{DeLathauwer2000},
\cite{Grasedyck2010} and \cite{Oseledets2009}, respectively for Tucker,
Hierarchical Tucker and Tensor Train tensors. Note that these
algorithms provide quasi-best approximations (but not necessarily best approximations) satisfying \eqref{eq:PiSX_eta} with a $\eta$ bounded by a function of
the dimension $d$: {$\eta\le \sqrt{d}$, $\eta\le\sqrt{2d-3}$ respectively for Tucker and Hierarchical Tucker formats (see \cite{Hackbusch2012}).  For a high dimension $d$, such bounds for $\eta$ would suggest taking very small values for parameter $\delta$ in order to satisfy the assumption of Proposition \ref{prop:gradient_algorithm_perturbed_convergence}. However, in practice, these a priori bounds are rather pessimistic. Moreover, quasi-best approximations obtained by higher order SVD 
can be used as initializations of optimization
algorithms yielding better approximations, i.e. with small values of $\eta$.} 

\paragraph{\emph{General norms, $d\ge 2$.}}
For a general norm $\Vert\cdot\Vert_X$, the computation of a global
optimum to the best approximation problem is still an open problem
for all tensor subsets, and methods based on SVD cannot be applied
anymore. However, classical optimization methods can still be
applied (such as Alternating Minimization Algorithm (AMA)) in order to
provide an approximation of the best approximation
\cite{Rohwedder2013,Uschmajew2012,Espig2012}. We do not detail further these
computational aspects and we suppose that algorithms are available
for providing an approximation of the best approximation in $\Sc_X$
such that \eqref{eq:PiSX_eta} holds with a controlled precision $\eta$, arbitrarily close to $1$.  

\subsection{{Construction of an approximation of $\Lambda^\delta(r)$}}\label{sec:comput_lambda_delta}
%\subsection{Application of mapping $\Lambda^\delta$}\label{sec:comput_lambda_delta}

At each iteration of the algorithm
\eqref{eq:gradient_algorithm_perturbed}, we have to compute $y^k =
\Lambda^\delta(r^k)$, with $r^k = R_Y^{-1}(Au^k-b)\in Y$, such that
it satisfies
\begin{align}
\Vert y^k - r^k\Vert_Y\le \delta \Vert r^k \Vert_Y.
\end{align}
%We here explain how this can be achieved. 
First note that $r^k $ is
the unique solution of
\begin{align}
\min_{r\in Y} \Vert r - R_Y^{-1}(Au^k-b)\Vert_Y^2.
\label{eq:auxiliary_problem_min}
\end{align}
Therefore, computing $y^k$ is equivalent to solving the best
approximation problem \eqref{eq:auxiliary_problem_min} with a
relative precision $\delta$.
% Noting that
% $$
% \Vert r - R_Y^{-1}(Au^k-b)\Vert_Y^2 = \Vert r \Vert_Y^2 - 2\langle
% Au^k-b,r\rangle_{Y',Y} + \Vert  R_Y^{-1}(Au^k-b) \Vert_Y^2,
% $$
% we have that 
One can equivalently characterize $r^k\in Y$ by the variational equation
$$
\langle r^k,\delta r\rangle_Y = \langle Au^k-b,\delta
r\rangle_{Y',Y}\quad \forall \delta r\in Y,
$$
or in an operator form:
\begin{align}
R_Y r^k = Au^k-b,\label{eq:auxiliary_problem_eq}
\end{align}
where the Riesz map $R_Y=AR_X^{-1}A^*$ is a positive symmetric
definite operator.

{
\begin{rmrk} For $A$ symmetric and positive definite, it is possible to choose  $R_X=R_Y=A$ (see example \ref{ex:symmetric}) that corresponds to the energy norm on $X$. For this choice, the auxiliary problem \eqref{eq:auxiliary_problem_min} has the same structure {as the initial problem}, with an operator $A$ and a right-hand side $Au^k-b$.
\end{rmrk}
}
\subsubsection{Low-rank tensor methods}

For solving \eqref{eq:auxiliary_problem_min}, we can also use low-rank
tensor approximation methods. Note that in general,
$\Vert\cdot\Vert_Y$ is not an induced (canonical) norm in $Y$, so
that classical tensor algorithms (e.g. based on SVD) cannot be
applied for solving \eqref{eq:auxiliary_problem_min} (even
approximatively). Different strategies have been proposed in the
literature for constructing tensor approximations of the solution of
optimization problems. {We can either use iterative solvers using classical tensor approximations applied to equation
\eqref{eq:auxiliary_problem_eq} \cite{Kressner2011,Khoromskij2011,Matthies2012,Ballani2013}, or directly 
compute an approximation  $y^k$ of $r^k$ in low-rank tensor subsets
%$\Sc_Y\subset Y$   
using optimization algorithms applied to problem
\eqref{eq:auxiliary_problem_min}. Here, we adopt the latter strategy and rely on a greedy algorithm which consists in computing successive corrections of the approximation in a fixed low-rank subset.}
%{ However, the latter strategy does not allow to achieve the required precision $\delta$ when the rank of the approximation is fixed. In order to achieve the precision, a first strategy would consist in constructing a succession of
%approximations $\{y^k_m\}_{m\ge 1}$ in an increasing sequence of
%tensor subsets $\{\Sc_Y^m\}_{m\ge 1}$ (with increasing rank) until the precision $\delta$
%is reached. Here, we rather adopt another  strategy
%which consists in keeping the subset $\Sc_Y$ fixed (possibly small)
%and in applying greedy algorithms for approximating $r^k$ ({see e.g. \cite{Falco2012-pgd}}).
%}

\subsubsection{A possible {(heuristic)} algorithm} \label{lambdadeltaalgo}

We use the following algorithm for the construction of a sequence of
approximations $\{y_m^k\}_{m\ge 0}$.\\

Let $y_0^k=0$. %and let $Z_0^k$ be an initial linear subspace. 
Then,
for each $m\ge 1$, we proceed as follows:
\begin{enumerate}
\item compute an optimal correction $w^k_m$ of $y_{m-1}^{k}$ in $\Sc_Y$:
$$
 w^k_m \in \arg\min_{w\in \Sc_Y} \Vert
 y_{m-1}^{k} + w - r^k \Vert_Y,
$$
\item {define a linear subspace $Z_m^k$ such that $y_{m-1}^k +w_m^k\in Z_m^k$},
%such that $Z_{m-1}^k\subset Z_m^k$ and $w^k_m\in Z_m^k$,
\item compute $y_m^k$ as the best approximation of $r^k$ in $Z_m^k$,
$$
y_m^k = \arg\min_{y\in Z_m^k} \Vert y-r^k\Vert_Y,
$$
\item return to step (2) or (1).
\end{enumerate}

{\begin{rmrk}
The convergence proof for this algorithm can be found in \cite{Falco2012-pgd}.
The convergence ensures that the precision $\delta$ can be achieved
after a certain number of iterations.\footnote{Note however that a
slow convergence of these algorithms may yield to high rank
representations of iterates $y^k_m$, even for a low-rank subset
$\Sc_Y$.} However, in practice, best approximation problems at step (1) can not be solved exactly except for particular situations (see section \ref{sec:best_approx_comput}), so that the results of  \cite{Falco2012-pgd} do not guaranty anymore the convergence of the algorithm. If quasi-optimal solutions can be obtained, this algorithm is a modified version of weak greedy algorithms  (see \cite{Temlyakov2008}) for which convergence proofs can also be obtained. { Available algorithms for obtaining quasi-optimal solutions of best low-rank approximation problem appearing at step (1) are still heuristic but seem to be effective.}
\end{rmrk}

 % RAJOUTER COMMENTAIRE SUR WEAK VERSIONS AND PROOF INSPIRED FROM TEMLYAKOV}\\

In this paper, we will only rely on the use of low-rank canonical formats for numerical illustrations. At step (1), we introduce rank-one corrections $w^k_m \in \Sc_Y = \Rc_1(Y)$, where $Y={}_{\Vert\cdot\Vert_Y} \bigotimes_{\mu=1}^d Y^\mu$. The
auxiliary variable $y_m^k \in \Rc_m(Y)$ can be {written in the form} $y_m^k = \sum_{i=1}^{m} \otimes_{\mu=1}^d w_i^{k,\mu} $. At step (2), we select a particular dimension $\mu \in \{1,\hdots,d\}$ and define 
$$
Z_m^k = \left\{ \sum_{i=1}^m w_i^{k,1}\otimes \cdots \otimes v_i^{\mu} \otimes \cdots\otimes w_i^{k,d} ~,~ v_i^{\mu} \in Y^{\mu} \right\},
$$
where $\text{dim}(Z_m^k) =m~\text{dim}(Y^{\mu})$. Step (3) therefore consists in updating functions $w_i^{k,\mu}$, $i=1\hdots d$, in the representation of $y_m^k$. Before returning to step (1), the updating steps (2)-(3) can be performed several times for a set of dimension $\mu \in I \subset \{1,\dots,d\}$.

\begin{rmrk}
Note that the solution of minimization problems at steps (1) and (3) do not require to know $r^k$ explicitly. Indeed, the {stationary conditions associated with} these optimization problems 
only require the evaluation of $\langle r^k,\delta y\rangle_{Y} = \langle Au^k-b,\delta y\rangle_{Y',Y}$, for $\delta Y\in Y$. {For step (1), the stationary equation reads} $\langle R_Y w_m^k,\delta y\rangle_{Y',Y}  = \langle R_Y y_{m-1}^k+  Au^k-b,\delta y\rangle_{Y',Y}$ for all $\delta y$ in the tangent space to $\mathcal{S}_Y$, while the variational form of step (3) {reads} $\langle R_Y y_m^k,\delta y\rangle_{Y',Y}  = \langle Au^k-b,\delta y\rangle_{Y',Y}$ for all $\delta y$ in $Z_m^k$.
%In the previous procedure, AMA (see section \eqref{sec:best_approx_comput}) are typically used for  
%step 1. Classically these algorithms rely on the variational equation \eqref{eq:auxiliary_problem_eq} that 
%avoid an explicite computation of $r^k$. 
\end{rmrk}

%
%In this paper, the linear spaces $Z_m^k$ of step (2) are constructed by means of successive updates in selected separation directions noted $\mu \in I$, with $I \subset \{1,\dots,d\}$.  Now assume that for given $m,k$ the auxiliary variable at step $m-1$ has a canonical low rank format i.e. 
%$y_{m-1}^k = \sum_{i=1}^{m-1} \otimes_{\mu=1}^d y_i^{k,\mu} \in Z_{m-1}^k$. Let initialize $y_{m}^k = y_{m-1}^k + w_m^1\otimes \dots \otimes w_m^d $. Then, for each $\mu \in I$, the proposed algorithm reads as follows :
%\begin{enumerate}
%\item[(i)] define 
%$$
%Z_m^k = \left\{ \sum_{i=1}^m y_i^{1}\otimes \cdots \otimes v_i^{\mu} \otimes \cdots\otimes y_i^{d} ~,~ v_i^{\mu} \in Y^{\mu} \right\},
%$$
%with $\text{dim}(Z_m^k) =m~\text{dim}(Y^{\mu})$ and $Y = \otimes_{\mu=1}^d Y_\mu$,
%% where $z_m^{\nu}=y_m^{k,\nu}$ for $\nu = 1,d$ and $\nu \neq \mu$,
%\item [(ii)] 
%then compute the best approximation of $r^k$ on $Z_m^k$ i.e.
%$$
%y_m^k = \arg \min_{y \in Z_m^k} \|y-r^k\|_Y.
%$$
%\end{enumerate} 
%}

%Further details on computational aspects for the solution
%of this type of problems can be found in e.g. \cite{Nou2010}.

Finally, as a stopping criterion, we use a heuristic error estimator based on stagnation.
The algorithm is stopped at iteration $m$ if 
\begin{equation}
e_m^p = \frac{\Vert y_{m}^{k}-y_{m+p}^{k}\Vert _Y}{\Vert y_{m+p}^{k}\Vert _Y} \leq \delta, \label{eq:stoppingcriterion}
\end{equation}
for some chosen $p\ge 1$ (typically $p=1$). Note that for $p$ sufficiently
large, $y_{m+p}^k$ can be considered as a good estimation of the
residual $r^k$ and the criterion {reads}
$\Vert r^k - y_{m}^k\Vert_Y \le \delta \Vert r^k\Vert_Y,$ which is the desired property. 
This stopping criterion is quite rudimentary and should be improved for a 
{real} control of the algorithm. {Although numerical experiments illustrate that this heuristic error estimator provides a rather good approximation of the true error, an upper bound of the true error should be used in order to {guarantee} that the precision $\delta$ is really achieved. However, a tight error bound should be used in order to avoid a pessimistic overestimation of the true error which may yield an (unnecessary) increase of the computational costs for the auxiliary problem. This key issue will be addressed in a future work.}
{
\begin{rmrk}
Other updating strategies could be introduced at steps (2)-(3). For example, we could choose $Z_m^k = \textrm{span}\{w_1^k,\hdots,w_m^k\}$, thus making the algorithm an orthogonal greedy algorithm with a dictionary $\Sc_Y$   \cite{Temlyakov2011}.  Nevertheless, numerical simulations demonstrate that when using rank-one corrections (i.e. $\Sc_Y=\Rc_1(Y)$), this updating strategy do not significantly improve the convergence of pure greedy constructions. When it is used for obtaining an approximation $y_m^k$ of $r^k$ with a small relative error $\delta$, it usually requires a very high rank $m$.
A more efficient updating strategy consists in defining 
$Z_m^k $ as the tensor space $\bigotimes_{\mu=1}^d Z_m^{k,\mu}$ with $Z_m^{k,\mu}$ $\textrm{span} \{w_1^{k,\mu},\hdots,w_m^{k,\mu}\}$. Since $\mathrm{dim}(Z_m^k)=m^d$, the projection of $r^k$ in  $Z_m^k$ can not be computed exactly for high dimensions $d$. However, approximations of this projection can be obtained using again low-rank formats (see \cite{Giraldi2012}). 
\end{rmrk}
}

\subsubsection{Remark on the tensor structure of Riesz maps}
We consider that operator $A$ and right-hand side  $b$ admit
low-rank representations
$$A =
\sum_{i=1}^{r_A} \otimes_{\mu=1}^d A_i^\mu \quad \text{and} \quad b
= \sum_{i=1}^{r_b} \otimes_{\mu=1}^d b_i^\mu.$$ We suppose that a
norm $\Vert\cdot\Vert_X$ has been selected and corresponds to a
Riesz map $R_X$ with a low-rank representation:
$$
R_{X} = \sum_{i=1}^{r_X}  \otimes_{\mu=1}^d R_i^\mu.
$$
The ideal choice of norm $\Vert\cdot\Vert_Y$ then corresponds to the
following expression of the Riesz map $R_Y$:
$$
R_Y = AR_X^{-1} A^* = (\sum_{i=1}^{r_A} \otimes_{\mu=1}^d A_i^\mu)
(\sum_{i=1}^{r_X} \otimes_{\mu=1}^d R_i^\mu)^{-1}(\sum_{i=1}^{r_A}
\otimes_{\mu=1}^d {A_i^\mu}^*).
$$
Note that the expression of $R_Y$ cannot be computed explicitly
($R_Y$ is generally a full rank tensor). Therefore, in the general
case, algorithms for solving problem
\eqref{eq:auxiliary_problem_eq} have to be able to handle an
implicit formula for $R_Y$. However,  in the  particular case where
the norm $\Vert \cdot\Vert_X$ is a canonical norm induced by norms $\Vert\cdot\Vert_{\mu}$ on $X_{\mu}$, the mapping $R_X$ is a rank one tensor $R_X=\otimes_{\mu=1}^d
R_{X_\mu}$, where $R_{X_\mu}$ is the Riesz map associated with the
norm $\Vert\cdot\Vert_\mu$ on $X_\mu$. $R_Y$ then admits the
following explicit expression:
$$
R_Y = A R_X^{-1} A^* = \sum_{i=1}^{r_A} \sum_{j=1}^{r_A}
\otimes_{\mu=1}^d (A_i^\mu R_{X_\mu}^{-1} {A_j^\mu}^*).
$$
In the numerical examples, we only consider this simple particular case. 
Efficient numerical methods for the general case will be proposed in a 
subsequent paper.

\subsection{Summary of the algorithm}

Algorithm \ref{alg:procedure1} provides  a step-by-step 
outline of the overall iterative method for the approximation  
of the solution of \eqref{eq:min_res_equivalent}
 in a fixed subset 
$\Sc_X$ and with a chosen metric $\Vert\cdot \Vert_{X}$. 
Given a precision $\delta$, an approximation of the residual is obtained with a greedy algorithm using a fixed subset $\Sc_Y$ for computing successive corrections. We denote by $e(y^k_m,r^k)$ 
an estimation of the relative error $\Vert y^k_m -r^k\Vert_Y/\Vert r^k\Vert_Y$, 
where $r^k=R_Y^{-1}(Au^k-b)$.

\begin{algorithm}[ht]
\begin{algorithmic}[1]
\STATE Set $u^0=0$;% and $Z_0^0=0$ ;
\FOR{$k=0$ to $K$} 
%\STATE Find $y^k =\Lambda^{\delta}(Au^k-b)$ following the procedure
%\STATE Compute the projection $ \displaystyle y_0^k = \arg \min_{y \in 
%Z_0^k} \Vert y-r^k\Vert_Y$ ;
\STATE Set $m=0$ ;
\WHILE{$e(y^k_m,r^k) \le \delta $}
\STATE $m=m+1$\;;
\STATE Compute a correction 
$ \displaystyle w_m^k \in \arg \min_{w\in \Sc_Y}\Vert y_{m-1}^k+w-r^k\Vert_Y$ ;
\STATE {Set $y_m^k = y_{m-1}^k+w_m^k$ ;}
\STATE {Define $Z_{m}^k$ containing $y_{m}^k$ ;}
\STATE Compute the projection $  \displaystyle y_m^k = \arg \min_{y \in Z_m^k}\Vert y-r^k\Vert_Y$ ;
\STATE {Return to step 7 or continue ; }
\ENDWHILE
\STATE Compute $u^{k+1} \in \Pi_{\Sc_X}^\eta(u^k-R_X^{-1}A^*y^k_m)$ ;
\ENDFOR
\end{algorithmic}
\caption{Gradient-type algorithm}
\label{alg:procedure1}
\end{algorithm}

% \IF{<condition>}
% \IF{<condition>} ba \ENDIF
% \FOR{<condition>} be \ENDFOR
% \FORALL{<condition>} bi \ENDFOR
% 
% \REPEAT <text> \UNTIL{<condition>}
% \REQUIRE <text>
% \RETURN <text>
% 

\section{Greedy algorithm}\label{sec:greedy}

In this section, we introduce and analyze a greedy algorithm for the
progressive construction of a sequence $\{u_m\}_{m\ge 0}$, where
$u_m$ is obtained by computing a correction of $u_{m-1}$ in a given
low-rank tensor subset $\Sc_X$ (typically a small subset such as the set of rank-one tensors  $\Rc_1(X)$). 
{Here, we consider that approximations of optimal corrections are
available with a certain precision. It results in an algorithm which can 
be considered as a modified version of weak greedy algorithms \cite{Temlyakov2008}.}
{This weak greedy algorithm can be applied to solve the best approximation problem
\eqref{eq:minimization_SX_residual} where approximations of optimal corrections are obtained using Algorithm \ref{alg:procedure1} with an updated right-hand side at each greedy step.  
The interest of such a global greedy strategy is twofold. First, an adaptive approximation strategy which would consist in solving approximation problems in an increasing sequence of low-rank subsets $\Sc_X$ is often unpractical since for high dimensional problems and subspace based tensor formats, computational complexity drastically increases with the rank. Second, 
it simplifies
the solution of auxiliary problems (i.e. the computation of the sequence of $y^k$) when solving best low-rank approximation problems using  Algorithm \ref{alg:procedure1}. Indeed, if the sequence $u^k$ in Algorithm \ref{alg:procedure1} belongs to a low rank tensor subset (typically a rank-one tensor subset), the residual $r^k$ in Algorithm \ref{alg:procedure1} admits a moderate rank or  can be obtained by a low-rank correction of the residual of the previous greedy iteration.}
 
Here, we assume that the subset $\Sc_X$ verifies properties
\eqref{SX_weakly_closed} and \eqref{SX_cone}, and that
$\mathrm{span}(\Sc_X)$ is dense in $X$ (which is verified by all
classical tensor subsets presented in section
\ref{sec:tensor_subsets}). 

\subsection{A weak greedy algorithm}
We consider the following greedy algorithm. Given $u_0=0$, we construct a sequence 
$\{u_m\}_{m\ge
1}$ defined for $m\ge 1$ by
\begin{align}
u_m = u_{m-1} + \widetilde w_m,\label{eq:pure_greedy}
\end{align}
where $\widetilde w_m\in\Sc_X$ is a correction of $u_{m-1}$ satisfying
\begin{align}
\Vert u-u_{m-1}-\widetilde w_m \Vert_X\le(1+\gamma_m)
\min_{w\in\Sc_X}\Vert u-u_{m-1}-w\Vert_X, \label{eq:condition_greedy_perturbation}
\end{align}
with  $\gamma_m$ a sequence of small parameters. 

\begin{rmrk}\label{rmrk:gamma_delta}
A $\widetilde w_m$ satisfying \eqref{eq:condition_greedy_perturbation} can
be obtained using the gradient type algorithm of section \ref{sec:perturbation} that provides
a sequence that satisfies \eqref{eq:sequence_uk_tends_to_neighborhood}. Given the parameter 
$\delta=\delta_m$ in \eqref{eq:gradient_algorithm_perturbed}, property
\eqref{eq:condition_greedy_perturbation} can be achieved with any
$\gamma_m > \frac{2\delta_m}{1-2\delta_m}$.
\end{rmrk}
%
%A $\widetilde w_m$ satisfying \eqref{eq:condition_greedy_perturbation} can
%be obtained using the algorithm of section \ref{sec:perturbation}
%(by replacing $u$ by $(u-u_m)$) with an appropriate choice of the
%precision $\delta=\delta_m$ and of the stopping criterion.
%\begin{rmrk}
%Note that algorithm \eqref{eq:gradient_algorithm_perturbed} provides
%a sequence that satisfies
%\eqref{eq:sequence_uk_tends_to_neighborhood}. Choosing
%$\delta=\delta_m$ in algorithm
%\eqref{eq:gradient_algorithm_perturbed}, property
%\eqref{eq:condition_greedy_perturbation} can then be achieved with any
%$\gamma_m > \frac{2\delta_m}{1-2\delta_m}$.
%\end{rmrk}
%

\subsection{Convergence analysis}\label{sec:convergence_greedy}

%Denoting 
%$f_{m-1}=u-u_{m-1}$, we then have 
%\begin{align} 
%%\Vert f_{m-1}-\Pi_{\Sc_X}(f_{m-1}) \Vert_X \le  
%\Vert %f_{m-1}-\widetilde w_m\Vert_X\le(1+\gamma _m) \Vert 
%f_{m-1}-\Pi_{\Sc_X}(f_{m-1})\Vert_X.
%\label{eq:greedy_quality_wtilde}
%\end{align} 
%Elements of proof are adapted from \cite{Temlyakov2009} and are given in
%appendix \ref{app:proof_greedy}. 
%The convergence conditions %will be discussed in the present context. 
Here, we provide a convergence result for the above greedy algorithm whose proof 
follows the lines of \cite{Temlyakov2008} for the convergence proof of weak greedy 
algorithms\footnote{Note that the condition
\eqref{eq:condition_greedy_perturbation} on the successive  corrections does
not  allow to directly apply the results on classical weak greedy  algorithms.}.

In the following, we denote by $f_{m}=u-u_m$. For the sake of simplicity,  we
denote by $\Vert\cdot\Vert=\Vert\cdot\Vert_X$ and 
$\langle\cdot,\cdot\rangle=\langle\cdot,\cdot\rangle_X$ and we  let $w_m\in
\Pi_{\Sc_X}(f_{m-1})$, for which we have the following useful relations  coming
from properties of best approximation problems in tensor subsets (see section
\ref{sec:tensor_subsets}): 
\begin{align} &\Vert f_{m-1} - w_m\Vert^2 = \Vert
f_{m-1}\Vert^2 - \Vert w_m \Vert^2 \quad \text{and} \quad  \Vert w_m\Vert^2 =
\langle f_{m-1},w_m  \rangle.\label{eq:rel6} 
\end{align} 
We introduce the
sequence  $\{\alpha_m\}_{m\ge 1}$  defined by 
\begin{equation} 
\alpha_m =
\dfrac{\Vert f_{m-1}-w_m\Vert}{\Vert f_{m-1}\Vert}\in[0,1[. \label{eq:alpha} 
\end{equation}
 {It can be also useful to introduce the computable sequence $\{\widetilde{\alpha}_m\}_{m\ge 1}$ such that
\begin{equation} 
\widetilde{\alpha}_m=
\dfrac{\Vert f_{m-1}-\widetilde{w}_m\Vert}{\Vert f_{m-1}\Vert}. \label{eq:alpha_tilde} 
\end{equation}
that satisfies for all $m\le 0$ 
\begin{equation} 
\alpha_m \le \widetilde{\alpha}_m \le (1+\gamma_m) \alpha_m. \label{eq:rel_alpha} 
\end{equation}
}

\begin{lmm}\label{lem:nfconverges}
Assuming that for all $m\geq1$ we have 
\begin{equation}\label{ass:lambda}
(1+\gamma_m)\alpha_m {<} 1,
\end{equation}
the sequence $\{\Vert f_m \Vert \}_{m\geq1}$ converges.
Furthermore, it is possible to define a positive sequence $\{\kappa_m\}_{m \ge 1}$ 
 as
\begin{equation}\label{eq:defkappa}
\kappa_m^2 = 2\frac{\langle f_{m-1},\widetilde w_m\rangle}{\Vert \widetilde w_m \Vert^2}-1,
\end{equation}
and we have $\{ \kappa_m\Vert \widetilde w_m\Vert \}_{m\geq1} \in \ell^2$.
\end{lmm}

\begin{proof} From \eqref{eq:pure_greedy} and
\eqref{eq:condition_greedy_perturbation}, we have  
$$     \Vert f_m\Vert =
\Vert f_{m-1} - \widetilde w_m\Vert \leq (1+\gamma_m) \Vert f_{m-1} - w_m\Vert
= (1+\gamma_m)\alpha_m \Vert f_{m-1}\Vert. 
$$ 
Under assumption
\eqref{ass:lambda}, $\{\Vert f_m\Vert\}_{m\geq1}$ is a {strictly} decreasing and positive 
sequence and therefore converges. {Moreover, this implies that 
$\widetilde {w}_m \neq 0$ and  since}    
\begin{equation*}   \Vert
f_{m-1} - \widetilde w_m\Vert^2      =\Vert f_{m-1}\Vert^2  -      
\left(2\langle f_{m-1},\widetilde w_m\rangle - \Vert      \widetilde w_m
\Vert^2\right) \leq \Vert f_{m-1}\Vert^2,   \end{equation*}   it follows that
$2\langle f_{m-1},\widetilde w_m\rangle {>}  \Vert \widetilde w_m \Vert^2$.
Therefore, $\kappa_m$ is {positive and} can be defined   by \eqref{eq:defkappa} and we
have   \begin{align*}\label{eq:propkappa}    \Vert f_{m-1} - \widetilde
w_m\Vert^2      &= \Vert f_{m-1}\Vert^2  - \kappa_m^2 \Vert \widetilde
w_m\Vert^2       = \Vert f_{0}\Vert^2  - \sum_{i=1}^{m}\kappa_i^2 \Vert
\widetilde w_i\Vert^2       ,   \end{align*} that completes the proof. \end{proof}

%At this step, let us point out some results to control the sequence
%$\kappa_m$, and give a link between $\Vert w_m\Vert $ and $\Vert \widetilde w_m\Vert $. 
We now provide a result giving a relation between $\Vert w_m\Vert $ and $\Vert \widetilde w_m\Vert $.
\begin{lmm}\label{lem:kappa_bound} Assume \eqref{ass:lambda} holds and let 
$\mu_m^2 =  \dfrac{1-(1+\gamma_m)^2\alpha_m^2}{1-\alpha_m^2}\in [0,1]$. Then, 
we have \begin{equation}\label{eq:bound_w_wt} \mu_m \Vert w_m\Vert  \leq \kappa_m
\Vert \widetilde w_m\Vert  \leq  \Vert w_m\Vert ,  \end{equation}      and
\begin{equation}\label{eq:bound_w_mu2} \frac{\mu_m}{2}\leq\kappa_m.
\end{equation}  
%  The parameter $\kappa_m$ is also bounded by the following inequality : 
%  \begin{equation}\label{eq:kappa_bound} 
%    \kappa_m\in \left[ \frac{1-\sqrt{1-\mu_m^2}}{\mu_m} , \frac{1+\sqrt{1-\mu_m^2}}{\mu_m} \right] 
%  \end{equation} 
%  A direct consequence is 
%  $$ 
%  \frac{\mu_m}{2}\leq\kappa_m
%  $$ 
\end{lmm}

\begin{proof}   From inequality \eqref{eq:condition_greedy_perturbation} and
from the optimality of $w_{m}$, it follows that
\begin{align*}
&\Vert f_{m-1}-w_m\Vert ^2 \leq\Vert f_{m-1}-\widetilde w_m\Vert ^2    \leq (1+\gamma_m)^2 \Vert f_{m-1}-w_m\Vert ^2 \\ 
\Rightarrow \;  &\Vert f_{m-1}\Vert ^2-\Vert w_m\Vert ^2 \leq \Vert f_{m-1}\Vert ^2-\kappa_m^2 \Vert \widetilde w_m\Vert ^2 \leq (1+ \gamma_m)^2 \alpha_m^2 \Vert f_{m-1}\Vert ^2\\
\Rightarrow\;&( 1-(1+\gamma_m)^2 \alpha_m^2 )\Vert f_{m-1}\Vert ^2 \leq \kappa_m^2    \Vert \widetilde w_m\Vert ^2 \leq \Vert w_m\Vert ^2
% \frac{1-(1+\gamma_m)^2 \alpha_m^2}{1-\alpha_m^2}\Vert w_m\Vert ^2 &\leq \kappa_m^2 \Vert \widetilde w_m\Vert ^2 \leq  \Vert w_m\Vert ^2.
\end{align*}
Using $\Vert f_{m-1}\Vert ^2=\Vert f_{m-1}-w_m\Vert ^2+\Vert w_m\Vert ^2 = \alpha_m^2    \Vert f_{m-1}\Vert^2
+\Vert w_m\Vert ^2 $, and using the definition of $\mu_m$,   we obtain
\eqref{eq:bound_w_wt}.    In addition, from the optimality of $w_m$, we have  
$ \langle \frac{\widetilde w_m}{\Vert \widetilde w_m\Vert } , f_{m-1} \rangle    \leq
\langle \frac{w_m}{\Vert w_m\Vert } , f_{m-1} \rangle = \Vert w_{m}\Vert,   $   or equivalently   $  
\frac{\kappa_m^2+1}{2} \Vert \widetilde w_m\Vert  \leq \Vert w_m\Vert .   $   Combined with
\eqref{eq:bound_w_wt}, it gives   $   \frac{\kappa_m^2+1}{2} \leq
\frac{\Vert w_m\Vert }{\Vert \widetilde w_m\Vert } \leq    \frac{\kappa_m}{\mu_m}$, which
implies   \eqref{eq:bound_w_mu2}. \end{proof}

\begin{prpstn}\label{lem:convergence_fm}
Assume \eqref{ass:lambda} and that $\left\{\mu_m^2\right\}_{m\ge1}$ is such that
$\sum_{m=1}^\infty  \mu_m^2= \infty. %, \label{eq:criteremu2}
$
 Then, if $\left\{f_m \right\}_{m\ge 1}$ converges, it converges to zero.
\end{prpstn}

\begin{proof} Let us use a proof by contradiction. Assume that $f_m \to f \neq
0$ as $m\to \infty$, with $f \in X$. As {$\mathrm{span}{(\Sc_X)}$ is dense in $X$}, there exists $\epsilon>0$ such that  $ \sup_{v \in \Sc_X}|\langle
f,\frac{v}{\Vert v\Vert }\rangle| \ge 2 \varepsilon. $  Using the definition of $w_m$
and of $f$ as a limit of $f_m$, we have that there exists $N>0$ such that 
\begin{equation}  \Vert w_m\Vert = \sup_{v \in \Sc_X}|\langle
f_{m-1},\frac{v}{\Vert v\Vert }\rangle| \ge \varepsilon,\quad \forall m \ge N. 
\end{equation} Thanks to 
\eqref{eq:bound_w_wt}, we  have \begin{align*} \Vert f_m\Vert ^2 &= \Vert f_{m-1}\Vert ^2 - \Vert 
\widetilde w_m \Vert ^2 \kappa^2_m \le \Vert f_{m-1}\Vert ^2 - \Vert  w_m \Vert ^2 \mu^2_m, \\     
     &\le \Vert f_N\Vert ^2     - \sum_{i=N+1}^m \mu_i^2 \Vert w_i\Vert ^2 \le \Vert f_N\Vert ^2 -     
      \varepsilon^2 \sum_{i=N+1}^m            \mu_i^2, \end{align*} which
implies that $\{\mu_m\}_{m\ge 0}\in\ell^2$, a contradiction to the
assumption.

\end{proof}

\begin{prpstn}\label{lmm:greedy_cv}
 Assume \eqref{ass:lambda}. Further assume that the sequence $\mu_m$ 
  is non increasing and verifies
  \begin{equation}
    \sum_{m=1}^\infty \frac{\mu_m^2}{m} = \infty \label{eq:condition_mui}.
  \end{equation}
Then the sequence $\{u_m\}_{m\ge 1}$ converges to $u$.
\end{prpstn}

\begin{proof}  Let two integers $n<m$ and consider   $$   \Vert f_n-f_m\Vert ^2 =
\Vert f_n\Vert ^2 -\Vert f_m\Vert ^2 -2\langle f_n-f_m,f_m \rangle.   $$ Defining $\theta_{n,m}=
|\langle f_n-f_m,f_m \rangle|$ and using Lemma \ref{lem:kappa_bound}, we obtain 
\begin{align*}     \theta_{n,m} &\leq
\sum_{i=n+1}^{m}  |\langle \widetilde w_i,f_m \rangle|    \leq  \Vert w_{m+1}\Vert 
\sum_{i=1}^{m} \Vert  \widetilde w_i \Vert  \leq  2\frac{\kappa_{m+1} \Vert  \widetilde
w_{m+1}\Vert }{\mu^2_{m+1}}   \sum_{i=1}^{m} \kappa_i \Vert  \widetilde w_i \Vert .
\end{align*} Lemma \ref{lem:nfconverges} implies that $\kappa_m\Vert  \widetilde
w_m\Vert  \in \ell^2$.  Together with assumption \eqref{eq:condition_mui}, and
using {Lemma 2.7}   in \cite{Temlyakov2011}, we obtain that   $  
\lim\inf_{m\rightarrow\infty}\max_{n<m} \theta_{n,m} = 0.   $  Lemma 2.8 in
\cite{Temlyakov2011} then proves that the sequence    $\{f_m\}_{m\ge 1}$ converges.
Noting that \eqref{eq:condition_mui} implies that    $\{\mu_m\}_{m=1}^\infty
\notin \ell^2$, Lemma \ref{lem:convergence_fm} allows    to conclude  the proof.
\end{proof}

{
In practice, condition \eqref{eq:condition_mui} can be satisfied by the following sufficient condition on the sequence $\widetilde{\alpha}_m$, which is a computable sequence.  
\begin{crllr}
 If there exists a constant $0<\epsilon <1$, independent of $m$, such that
 \begin{equation}
 \widetilde{\alpha}^2_m  \leq \frac{1-\epsilon}{(1+\gamma_m)^2-\epsilon}, \label{eq:final_condition}
 \end{equation}
 then the sequence $\{u_m\}_{m\ge 1}$ converges to $u$.
\end{crllr}
}
{
\begin{proof}
Under assumption \eqref{eq:final_condition}  and using relation \eqref{eq:rel_alpha}, it holds that for all $m\geq0$
$$
  \alpha^2_m  \leq  \frac{1-\epsilon}{(1+\gamma_m)^2-\epsilon}\quad  \Rightarrow  \quad 
 (1+\gamma_m)^2\alpha_m^2 \leq 1-\epsilon(1-\alpha_m^2)<1.
 $$
%$$
% \begin{array}{rrcl}
%  		     & \widetilde{\alpha}^2_m \Big((1+\gamma_m)^2-\epsilon\Big) &\leq& (1-\epsilon) \\[0.2cm]
%  \Rightarrow  & \alpha^2_m \Big((1+\gamma_m)^2-\epsilon\Big) &\leq& (1-\epsilon) \\[0.2cm]
%  \Rightarrow& (1+\gamma_m)^2\alpha_m^2 &\leq& 1-\epsilon(1-\alpha_m^2)<1.
% \end{array}
% $$
which implies condition \eqref{ass:lambda}. Moreover, we have 
$$
\epsilon(1-\alpha_m^2) \leq 1-(1+\gamma_m)^2\alpha_m^2 \quad  \Rightarrow \quad \epsilon \leq \frac{1-(1+\gamma_m)^2\alpha_m^2}{(1-\alpha_m^2)} = \mu_m^2,
$$
which implies  condition \eqref{eq:condition_mui}. Proposition \ref{lmm:greedy_cv} ends the proof.
\end{proof}
}
%{
%\begin{rmrk}\label{rmk:final_condition}
% At each iteration, it is possible to set the parameter $\gamma_m$ so that the condition \eqref{eq:final_condition} is satisfied.
% Indeed, we have:
% \begin{equation*}
%  \Vert f_{m-1} - \widetilde w_m \Vert \sqrt{(1+\gamma_m)^2-C} \leq \Vert f_{m-1} - w_m \Vert (1+\gamma_m) \sqrt{(1+\gamma_m)^2-C}
% \end{equation*}
% Note that with a small value of $\gamma_m$, the term $(1+\gamma_m) \sqrt{(1+\gamma_m)^2-C}$ can be arbitrarily close to $\sqrt{1-C}$.
% In particular one can have $(1+\gamma_m) \sqrt{(1+\gamma_m)^2-C}\leq\sqrt{1-C}/\alpha_m$, what implies \eqref{eq:final_condition}.
%\end{rmrk}
%}
{
\begin{rmrk}
 From a practical point of view, condition \eqref{eq:final_condition} provides a sufficient criterion on $\gamma_m$ (or equivalently on $\delta_m$). Note that $\widetilde \alpha_m$ depends on $\widetilde w_m$ which depends on the choice of the precision $\gamma_m$. Therefore, \eqref{eq:final_condition} is an implicit condition on $\gamma_m$ which suggests an iterative strategy for the control of the condition. A possible strategy would be to adapt the parameter $\gamma_m$  during the iterations of the gradient type algorithm used to compute the $\widetilde w_m$.
\end{rmrk}
}

%%%%%%%%%%%
%%% PARTIE NUM
%%%%%%%%%%%

\section{Numerical example}\label{sec:example}

In this section, we apply the proposed method to the numerical solution 
of  a stochastic steady reaction-advection-diffusion problem. 
%The ideal minimal residual method is used for the 
%After a discretization using spatial finite elements and polynomial chaos expansions, 
%The ideal minimal residual method is applied to a discretized 
%First, a discretization of the problem is given combining the finite element method for space variables to polynomial chaos for %stochastic part. Further to this step, the proposed minimal residual procedure in tensor sets formulated for different norms $\Vert %\cdot\Vert _X$ is applied to the resulting algebraic system.
%Finally, details of computational aspects are given in a third section.

\subsection{Stochastic reaction-advection-diffusion problem} \label{sec:problem_setting}
{
We consider the following steady reaction-advection-diffusion problem on a two-dimensional unit square domain $\Omega =[0,1]^2$ (see Figure \ref{fig:geometry}):
\begin{align}
  -\nabla \cdot (\kappa \nabla u)+ c \cdot \nabla  u +a u= f \quad&\text{ in } \Omega, \label{eq::addiffre}\\
u=0 \quad &\text{ on } \partial\Omega. \nonumber
\end{align}
 First, we consider a constant diffusion $\kappa=1$}. The advection coefficient $c$ and the reaction coefficient $a$  are considered as random and are given by 
$c= \xi_{1} c_{0}$ and $a=\exp(\xi_2)$, where $\xi_{1} \sim U(-350,350)$ and $\xi_{2}\sim U(\log(0.1),\log(10))$  are independent uniform random variables, and 
 $c_{0}(x)=(x_2-1/2 , 1/2-x_1)$, $x=(x_{1},x_{2})\in \Omega$. We denote by $\Xi_{1} = ]\text{-}350,350[$ and $\Xi_{2}=]\log(0.1),\log(10)[$, and we denote by $(\Xi,\mathcal{B}(\Xi),P_{\xi})$ the probability space induced by $\xi=(\xi_{1},\xi_{2})$, with $\Xi=\Xi_{1}\times \Xi_{2}$ and $P_{\xi}$ the probability law of $\xi$. The external source term $f$ is given by 
$
f(x) = I_{\Omega_{1}}(x) - I_{\Omega_{2}}(x),
$ where $\Omega_1= ]0.45,0.55[\times]0.15,0.25[$ and $\Omega_2=]0.45,0.55[\times]0.75,0.85[$, and where $I_{\Omega_{k}}$ denotes the indicator function of $\Omega_{k}$. 

\begin{figure}[h]
  \centering
  \includegraphics[scale=0.75]{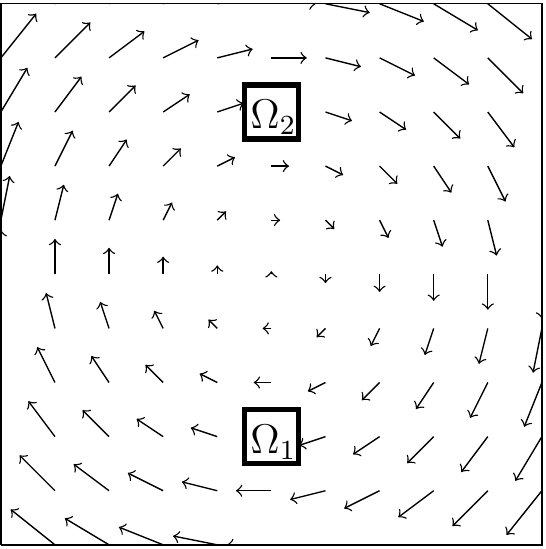}
  \caption{Example : reaction-advection-diffusion problem.}
  \label{fig:geometry}
\end{figure}
Let ${V}=\xHone_{0}(\Omega)$ and ${S}=\xLtwo(\Xi,dP_{\xi})$. We introduce 
approximation spaces $V_N\subset V$ and $S_P\subset S$, with $N=\dim(V_N)$ and $P=\dim(S_P)$.
$V_N$ is a $\mathbb{Q}_1$ finite element space associated with a uniform mesh of 1600 elements such that $N=1521$.
We choose $S_{P} = S^{\xi_1}_{p_1}\otimes S^{\xi_2}_{p_2}$, where $S^{\xi_1}_{p_1}$ is the space of piecewise polynomials of degree 5 on $\Xi_{1}$ associated with the partition $\{]\text{-}350,0[,]0,350[\}$ of $\Xi_{1}$, and $S^{\xi_2}_{p_2}$ is the space of polynomials of degree $5$ on $\Xi_{2}$. This choice results in $P=72$. The Galerkin approximation $u\in V_{N}\otimes S_{P} \subset V\otimes S$ of the solution of \eqref{eq::addiffre} is defined by the following equation\footnote{The mesh P\'eclet number is sufficiently  small so that an accurate Galerkin approximation can be obtained without introducing a stabilized formulation.}:
\begin{equation}
\int_{\Xi}\int_{\Omega} \left( \nabla  u \cdot\nabla  v +  c \cdot \nabla u v+ a  uv\right) dx~dP_{\xi}  = 
\int_{\Xi}\int_{\Omega} f~ v~dx~dP_{ \xi},\label{eq:FV}
\end{equation}
for all $v\in V_{N}\otimes S_{P}$. Letting $V_{N}\otimes S_{P}=span\{\varphi_{i}\otimes \psi_{j};1\le i\le N,1\le j\le P\}$, the Galerkin 
approximation  $u = \sum_{i=1}^{N}\sum_{j=1}^{P} u_{ij} \varphi_{i}\otimes \psi_{j} $ can be identified with its set of coefficients on the chosen basis, still denoted $u$, which is a tensor  
\begin{equation}
u\in X=\xR^N \otimes \xR^P \quad \text{such that} \quad Au=b,\label{eq:algebrique}
\end{equation}
where $b = b^{x}\otimes b^{\xi}$, with $b^{x}_{i}=\int_{\Omega} f\varphi_{i}$ and $b^{\xi}_{j}=\int_{\Xi} \psi_{j} dP_{\xi}$, and where $A$ is a rank-3 operator such that
$A = D^{x}\otimes M^{\xi} + C^{x} \otimes H^{\xi_{1}} + R^{x}\otimes H^{\xi_{2}}$, with $D^{x}_{ik} = \int_{\Omega}\nabla \varphi_{i} \cdot\nabla \varphi_{k} dx$,  $C^{x}_{ik} = \int_{\Omega}  \varphi_{i} c_{0}\cdot\nabla \varphi_{k}  dx$, $R^{x}_{ik} = \int_{\Omega}  \varphi_{i} \varphi_{k}  dx$, $M^{\xi}_{jl} =\int_{\Xi}\psi_{j}(y)\psi_{l}(y) dP_{\xi}(y)$, $H^{\xi_{n}}_{jl} = \int_{\Xi}y_{n}\psi_{j}(y)\psi_{l}(y) dP_{\xi}(y)$,  $n=1,2$. Here, we use orthonormal basis functions $\{\psi_{j}\}$ in $S_{P}$, so that $M^{\xi}=I_{P}$, the identity matrix in $\mathbb{R}^{P}$.

%\begin{rmrk}
%The spatial mesh is chosen so that the mesh P\'eclet number is sufficiently  small in comparison to one.
%In that way, we do not require the need of stabilization for spatial discretization.
%In the contrary case, the task of numerical instability should be avoided using e.g. classical stabilized finite element approximation (SUPG) or specific finite element (bubble).
%% Another alternative, should be to take part of the continuous framework and formulate ideal variational formulation leading to ideally conditioned discrete problem as proposed in section 4.
%% In that case, we would propose ideal tensorised approximations in $\Rc_r(V_N\otimes S_P)$ for an ideal metric issue from the continuous operator. 
%\end{rmrk}

%% PLUS LOIN ? 
%\begin{rmrk}
%In that context, 
%\end{rmrk}

\subsection{Comparison of minimal residual methods} \label{sec:comp_minres}
In this section, we present numerical results concerning the approximate ideal minimal residual method (A-IMR) applied to  the algebraic system of equations \eqref{eq:algebrique} in tensor format. This method provides an approximation of the best approximation of $u$ with respect to a norm $\Vert\cdot\Vert_{X}$ that can be freely chosen a priori. Here, we consider the application of the method for two different norms. 
We first consider the natural canonical norm on $X$, denoted $\Vert \cdot\Vert _2$ and defined by 
\begin{equation}
\Vert v\Vert _2^{2}= {\sum_{i=1}^N\sum_{j=1}^P (v_{ij})^2  }.
\end{equation}
This choice corresponds to an operator $R_X=I_{X}=I_N\otimes I_P$, where $I_N$ (resp. $I_P$) is the identity in $\mathbb{R}^{N}$ (resp. $\mathbb{R}^{P}$). We also consider a weighted canonical norm, denoted $\Vert \cdot\Vert _w$ and defined by
\begin{equation}\label{eq:weighted_norm}
\Vert v\Vert _w^{2}= { \sum_{i=1}^N\sum_{j=1}^P \left( w(x_i) v_{ij} \right)^2  },
\end{equation}
where $w:\Omega\rightarrow \xR$ is a weight function and the $x_{i}$ are the nodes associated with finite element shape functions $\varphi_{i}$. This norm allows to give a more important weight to a particular region $D\subset \Omega$, that may be relevant if one is interested in the prediction of a quantity of interest that requires a good precision of the numerical solution in  this particular region (see section \ref{sec:interest_weigthed_norm}). 
This choice corresponds to an operator $R_X=D_{w}\otimes I_P$, with $D_{w}=\mathrm{diag}(w(x_1)^2,\hdots,w(x_{N})^{2})$.

The A-IMR provides an approximation $\widetilde u \in \Sc_X$ of the $\Vert \cdot\Vert _X$-best approximation of $u$ in $\Sc_{X}$ (that means an approximation of an element in $\Pi_{\Sc_{X}}(u)$), where 
 $\Vert \cdot\Vert _X$ is either $\Vert \cdot\Vert _2$ or $\Vert \cdot\Vert _w$.  The set $\Sc_{X}$ is taken as the set $\mathcal{R}_r(X)$ of rank-$r$ tensors in $X=\mathbb{R}^{N}\otimes \mathbb{R}^{P}$.
The dimension of $X$ is about 75,000 so that the exact solution $u$ of \eqref{eq:algebrique}  can be computed and used as a reference solution. 
We note that both norms are induced norms in $\mathbb{R}^{N}\otimes \mathbb{R}^{P}$ (associated with rank one operators $R_{X}$) so that the 
$\Vert \cdot\Vert _X$-best approximation of $u$ in $\Sc_{X}$ is a rank-$r$ SVD that can be computed exactly using classical algorithms (see section \ref{sec:best_approx_comput}).\footnote{Note that different truncated SVD are obtained when $\mathbb{R}^{N}$ is equipped with different norms.} 
For the construction of an approximation in $\mathcal{R}_r(X)$ using A-IMR, we consider two strategies: the
 direct approximation in $\mathcal{R}_r(X)$ using Algorithm \ref{alg:procedure1} with $\mathcal{S}_{X}=\mathcal{R}_r(X)$, and a greedy algorithm that consists in a series of $r$ corrections in $\mathcal{R}_1(X)$ computed using 
 Algorithm \ref{alg:procedure1} with $\mathcal{S}_{X}=\mathcal{R}_1(X)$ and with an updated residual $b$ at each correction.

The A-IMR will be compared to a standard approach, denoted CMR, which consists in minimizing the canonical norm of the residual of equation \eqref{eq:algebrique}, that means in solving 
 \begin{equation}
  \min_{v \in \Sc_X} \Vert Av-b\Vert _2.
  \label{eq:minresCMR}
\end{equation}
This latter approach has been introduced and analyzed in different papers, using either direct minimization or greedy rank-one algorithms \cite{Beylkin2005,Doostan2009,Ammar2010},  and is known to suffer from ill-conditioning of the operator $A$.
We note that this approach corresponds to choosing $R_X=A^*A$ and $R_{Y} = I_{X}=  I_{N}\otimes I_{P}$.

\subsubsection{Natural canonical norm $\Vert \cdot\Vert _2 $}

First, we compare both greedy and direct algorithms for $\Vert \cdot\Vert _X=\Vert \cdot\Vert _2$, using either CMR or A-IMR with different precisions $\delta$. The convergence curves with respect to the  rank are shown in Figure \ref{fig:result_direct_greedy_NC}, where the error is measured in the $\Vert \cdot\Vert _2$ norm.
Concerning the direct approach, we observe that the different algorithms have roughly the same rate of convergence. The A-IMR convergence curves are close to the optimal SVD {(corresponding to $\widetilde u_2$)} for a wide range of values of $\delta$. One should note that A-IMR seems to provide good approximations also for the value $\delta =0.9$ which is greater than the theoretical bound $0.5$ ensuring the convergence of the gradient-type algorithm.
Concerning the greedy approach, we observe a significant difference between A-IMR and CMR. We note that A-IMR is close to the optimal SVD up to a certain rank (depending on $\delta$) after which the convergence rate decreases but remains better than the one of CMR. Finally, one should note that using a precision $\delta=0.9$ for A-IMR yields less accurate approximations than CMR. However, A-IMR provides better results than CMR once the precision $\delta$ is lower than  $0.5$.

\begin{figure}[h]
  \centering
  \includegraphics[scale=0.85]{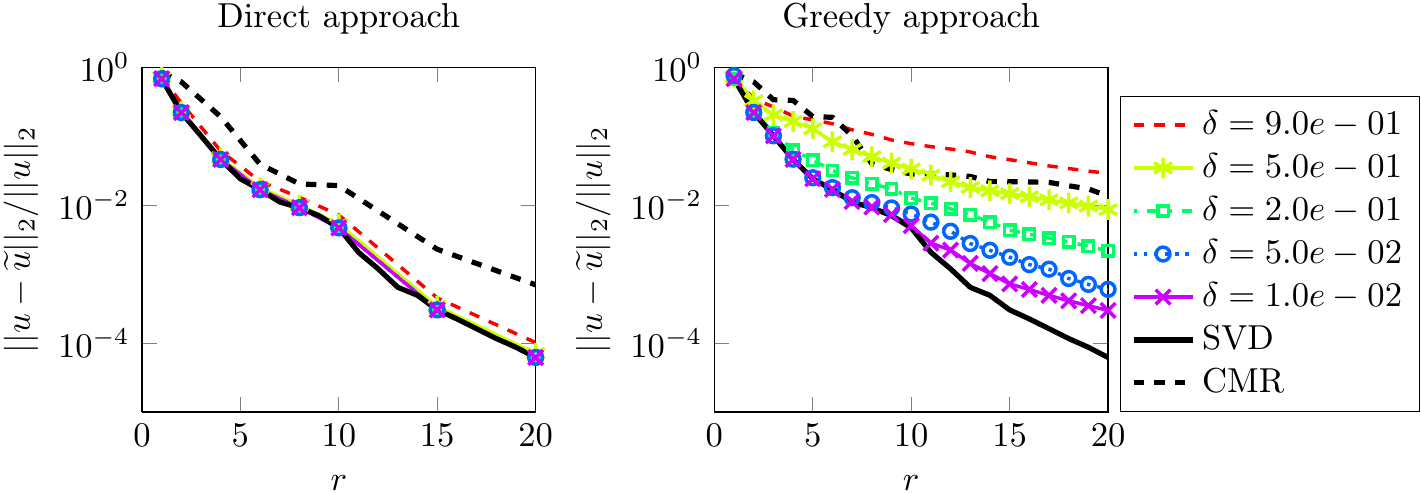}
  \caption{Comparison of minimal residual methods for $\Sc_X= \Rc_r(X)$ and $\Vert \cdot\Vert _X=\Vert \cdot\Vert _2$. Convergence with the rank $r$ of the approximations obtained with CMR or A-IMR with different precisions $\delta$, and with direct  (left) or greedy rank-one (right) approaches. }
  \label{fig:result_direct_greedy_NC}
\end{figure}

\subsubsection{Weighted norm $\Vert \cdot\Vert _w $ }

Here, we perform the same numerical experiments as previously using the weighted norm $\Vert \cdot\Vert _X=\Vert \cdot\Vert _w$, with $w$ equal to $10^3$ on $D=[0.15,0.25]\times[0.45,0.55]$ and $w=1$ on $\Omega\setminus D$. 
The convergence curves with respect to the  rank are plotted on Figure  \ref{fig:result_direct_greedy_NW}. 
The conclusions are similar to the case $\Vert \cdot\Vert _X=\Vert \cdot\Vert _2$, although  the use of the weighted norm seems to slightly deteriorate the convergence properties of A-IMR. However, the direct A-IMR still provides better approximations than the direct CMR, closer to the reference SVD {(denoted by $\tilde{u}_w$)} for different values of precision $\delta$.

\begin{figure}[h]
  \centering
  \includegraphics[scale=0.85]{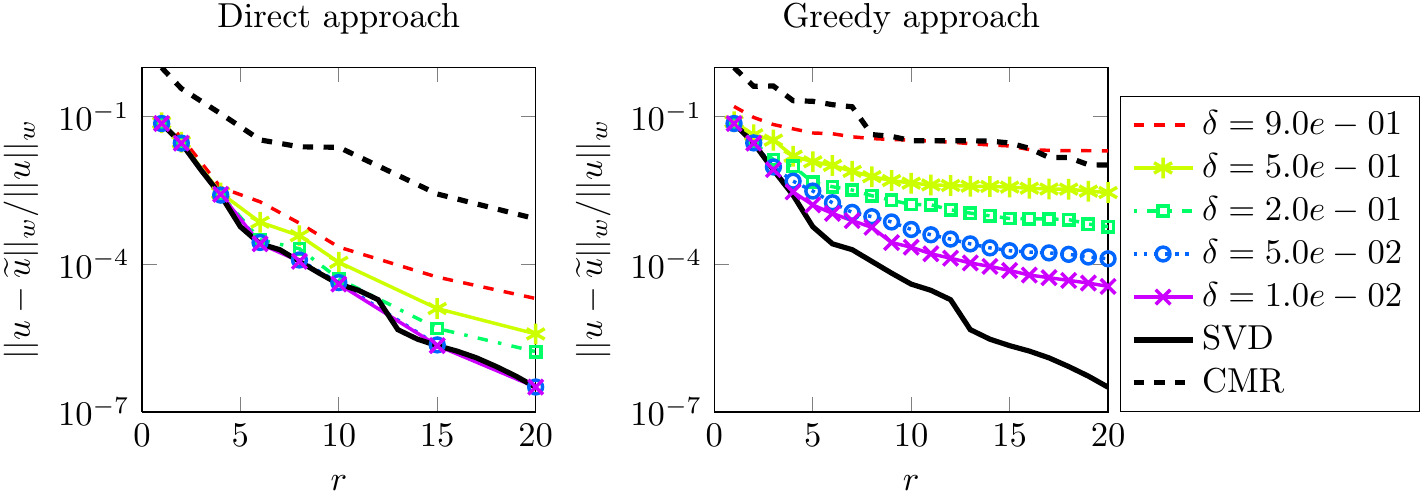}
  \caption{Comparison of minimal residual methods for $\Sc_X= \Rc_r(X)$ and $\Vert \cdot\Vert _X=\Vert \cdot\Vert _w$. Convergence with the rank of the approximations obtained with CMR or A-IMR with different precisions $\delta$, and with direct  (left) or greedy rank-one (right) approaches. }
  \label{fig:result_direct_greedy_NW}
\end{figure}

\subsubsection{Interest of using a weighted norm} \label{sec:interest_weigthed_norm}

Here, we illustrate the interest of using the weighted  norm rather than the natural canonical norm when one is interested in computing a quantity of interest. For the sake of readability, we let $\widetilde u_w$ (resp. $\widetilde u_2$) denote the best approximation of $u$ in $\Rc_r({X})$ with respect  to the norm $\Vert \cdot\Vert _w$ (resp. $\Vert \cdot\Vert _2$).
Figure \ref{fig:compare_metric}  illustrates the convergence with $r$ of these approximations. We observe that approximations $\widetilde u_w$ and $\widetilde u_2$ are of the same quality when the error is measured with the norm $\Vert \cdot\Vert_{2}$, while $\widetilde u_w$ is a far better approximation than $\widetilde u_2$ (almost two orders of magnitude) when the error is measured with the norm $\Vert\cdot\Vert_{w}$. We observe that $\widetilde u_w$ converges faster to $u$ with $\Vert \cdot\Vert _w$ than $\widetilde u_2$ with $\Vert \cdot\Vert _2$.
For example, with a rank $r=9$, $\widetilde u_w$ has a $\Vert \cdot\Vert _w$-error of $10^{4}$ while $\widetilde u_2$ has a $\Vert \cdot\Vert _2$-error of $10^{2}$. On Figure  \ref{fig:Modes}, plotted are the spatial modes of the rank-$r$ approximations 
 $\widetilde u_2$ and $\widetilde u_w$. These spatial modes are significantly different and obviously capture different features of the solution.
 
 \begin{figure}[h]
  \centering
  \includegraphics[scale=0.9]{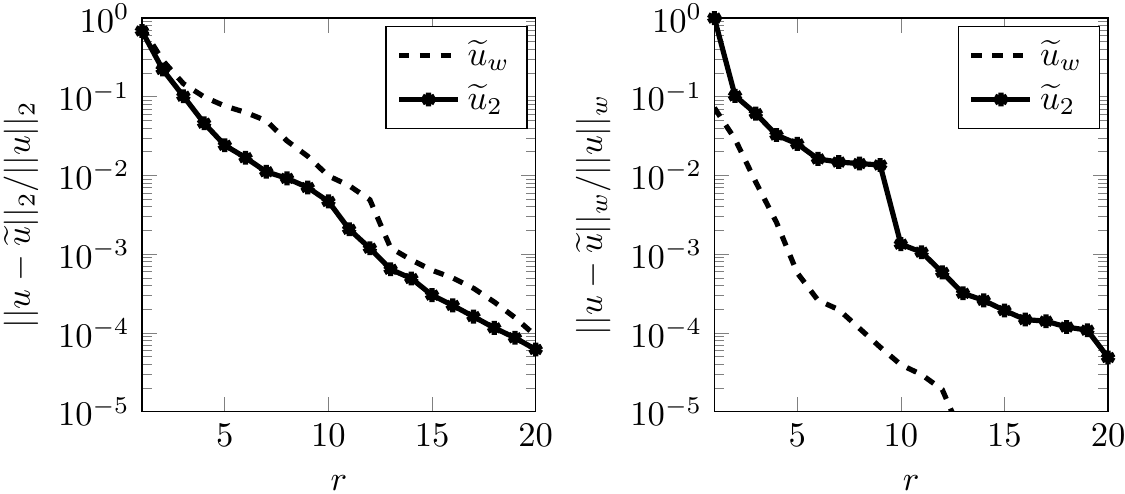}
  \caption{Convergence of best rank-$r$ approximations $\widetilde u_{2}$ and $\widetilde u_w$ of the solution $u$ measured with the natural canonical norm $\Vert\cdot\Vert_{2}$ or the weighted norm $\Vert\cdot\Vert_{w}$.}
  \label{fig:compare_metric}
\end{figure}

\begin{figure}[h]
  \centering
  \includegraphics[width=0.8\textwidth]{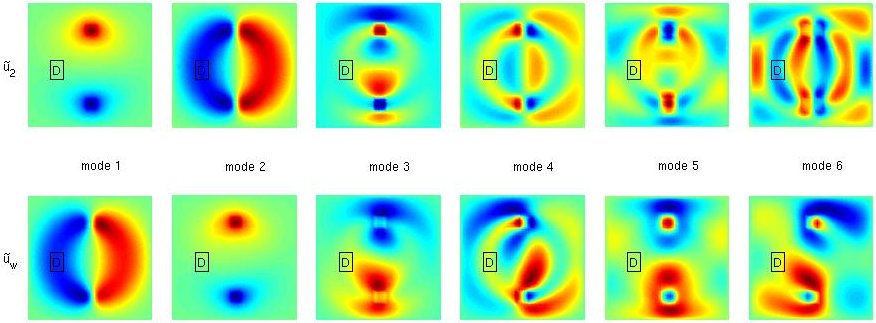}
  \caption{Comparison of the first spatial modes of the rank-$r$ approximations $\widetilde u$ and $\widetilde u_w$.}
  \label{fig:Modes}
\end{figure}

Now, we introduce a quantity of interest $Q$ which is the spatial average of $u$ on subdomain $D$:
\begin{equation}
Q(u) =  \frac{1}{|D|} \int_D u~\text{d}x.
\label{eq:QOIdef}
\end{equation}
Due to the choice of norm, $\widetilde u_w$ is supposed to be more accurate than $\widetilde u_{2}$ in the subdomain  $D$, and therefore, $Q(\widetilde u_w)$  is supposed to provide a better estimation of $Q(u)$ than {$Q(\widetilde u_2)$}. 
This is confirmed by Figure  \ref{fig:convergence_Q}, where we have plotted the convergence with the rank of the statistical mean and variance of $Q(\widetilde u_w)$ and $Q(\widetilde u_2)$.  With only a rank $r=5$, $\widetilde u_w$ gives a precision  of $10^{-7}$ on the mean, whereas $\widetilde u_2$ gives only a precision of $10^{-2}$. In conclusion, we observe that a very low-rank approximation $\widetilde u_{w}$ is able to provide a very good approximation of the quantity of interest.

\begin{figure}[h]
   \centering
   \subfigure[Error on the mean value of $Q(\widetilde u)$]{
     \centering
     \includegraphics[scale=0.95]{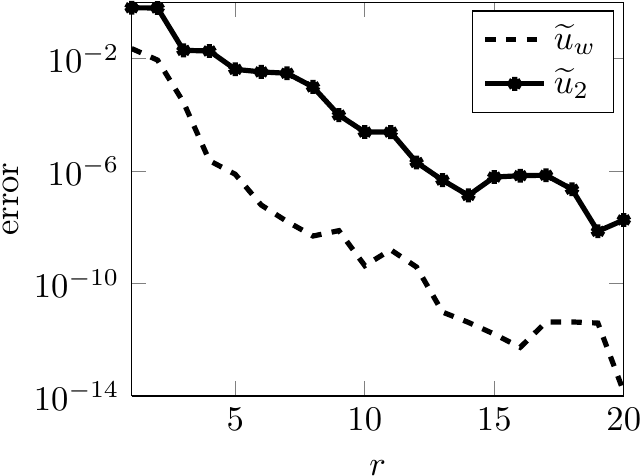}
     }
   \subfigure[Error on the variance of $Q(\widetilde u)$]{
     \centering
     \includegraphics[scale=0.95]{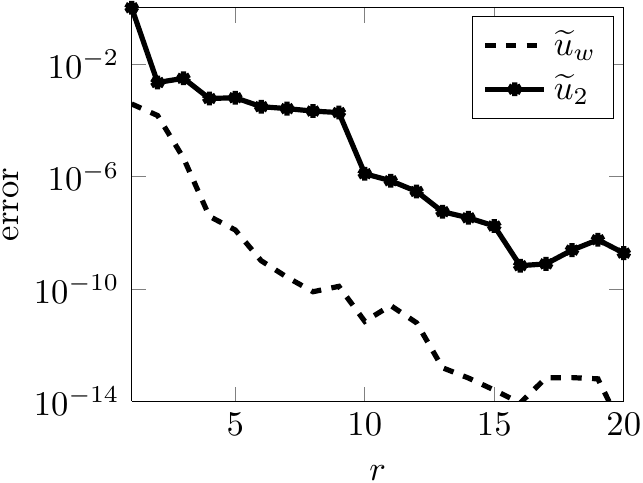}
     }
   \caption{Convergence with the rank of  the mean (left) and variance (right) of $Q(\widetilde u_2)$ and $Q(\widetilde u_w)$. Relative error with respect to the mean and variance of the reference solution $Q(u)$.}
   \label{fig:convergence_Q}
\end{figure}

\subsection{Properties of the algorithms}

Now, we detail some numerical aspects of the proposed methodology.
We first focus on the gradient-type algorithm, and then on evaluations of the map $\Lambda^{\delta}$ for the approximation of residuals.

\subsubsection{Analysis of the gradient-type algorithm}\label{sec:analysis_gradient}

%We first analyze the behavior of the gradient-type algorithm and give an estimate of the related approximation error.
%{ Note that as long as the $\delta$-proximality is satisfied (see section \ref{sec:construction_lambda}), the presented results holds whatever $\Vert \cdot\Vert _X$ is the 2-norm or the weighted 2-norm.}
% rajouter l'autre ou pas pour s'en convaincre avec les courbes ?  That's the question.
The behavior of the gradient-type algorithm for different choices of norms $\Vert \cdot\Vert_{X}$ is very similar, so we only illustrate the case where $\Vert \cdot\Vert_{X}= \Vert \cdot\Vert_{2}$.
The convergence of this algorithm is plotted in Figure \ref{fig:gradient_iter_cv} for the case $\Sc_{X}=\Rc_{10}(X)$. It is in very good agreement with theoretical expectations (Proposition \ref{prop:gradient_algorithm_perturbed_convergence}):
we first observe a linear convergence with a convergence rate that depends on $\delta$, and then a stagnation within a neighborhood of the solution with an error depending on $\delta$.
\begin{figure}[h]
  \centering
  \includegraphics[scale=0.85]{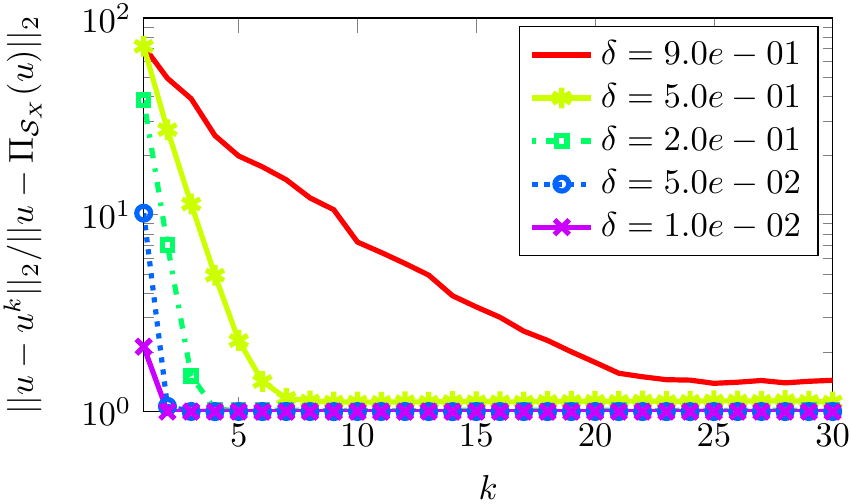}
  \caption{Convergence of the gradient-type algorithm for different values of the relative precision $\delta$, for $\mathcal{S}_X=\mathcal{R}_{10}(X)$ and $\Vert\cdot\Vert_{X}=\Vert\cdot\Vert_{2}$.}
  \label{fig:gradient_iter_cv}
\end{figure}
The gradient-type algorithm is then applied for subsets $\Sc_X=\mathcal{R}_{r}(X)$ with different ranks $r$. The estimate 
of the linear convergence rate $\rho$ is given in Table \ref{Tab1}.
We observe that for all values of $r$, $\rho$ takes values closer to $\delta$ than to the theoretical bound $2\delta$ of Proposition \ref{prop:gradient_algorithm_perturbed_convergence}.
This means that the theoretical bound of the convergence rate overestimates the effective one, and the algorithm converges faster than expected.
\begin{table}[h]
\begin{center}
% \small
\footnotesize
\begin{tabular}{|c|c|c|c|c|c|}\hline 
 $\delta$   &0.90    &0.50    &0.20       &0.05       &0.01       \\ \hline
 $r=4$      &0.78    &0.36    &$\approx 0$&$\approx 0$&$\approx 0$\\
 $r=6$      &0.83    &0.45    &0.165      &$\approx 0$&$\approx 0$\\
 $r=10$     &0.82    &0.42    &0.183      &$\approx 0$&$\approx 0$\\
 $r=15$     &0.84    &0.47    &0.189      &0.047      &$\approx 0$\\
 $r=20$     &0.86    &0.48    &0.197      &0.051      &0.011      \\ \hline
\end{tabular}\normalsize
\end{center}
\caption{Estimation of the convergence rate $\rho$ of the gradient-type algorithm (during the linear convergence phase) for different subsets $\Sc_X=\Rc_r(X)$, and for $\Vert \cdot\Vert _X=\Vert \cdot\Vert _2$.}
\label{Tab1}
\end{table}
Now, in order to evaluate the quality of the resulting approximation, we compute the error after the stagnation phase has been reached. More precisely, we compute the  value
$$
\widetilde\gamma_{k} = \frac{\Vert u^{k}-u\Vert _X}{\Vert u-\Pi_{\mathcal{S}_X}(u)\Vert _X} -1,
$$
for $k=100$.  Values of $\widetilde\gamma_{100}$
are summarized in Table \ref{Tab2} and are compared to the theoretical upper bound $\gamma=2\delta/(1-2\delta)$ 
given by Proposition \ref{prop:gradient_algorithm_perturbed_convergence}. Once again, one can observe that the effective error of the resulting approximation is lower than the predicted value regardless of the choice of $\Rc_r(X)$.
\begin{table}[h]
\begin{center}
% \small
\footnotesize
\begin{tabular}{|c|c|c|c|c|c|c|c|c|} \hline 
$\delta$                &  0.90    &  0.50   &  0.20   &  0.05   &0.01   \\ \hline
$2\delta/(1-2\delta)$   &     -    &    -    &6.6e-1   &1.1e-1   &2.1e-2 \\ \hline
$r=4$			&  3.3e-1  &5.6e-2   &4.9e-3   &3.5e-4   &3.0e-5 \\
$r=6$			&  3.0e-1  &6.8e-2   &1.1e-2   &8.6e-4   &8.0e-5 \\
$r=10$			&  5.2e-1  &1.3e-1   &1.7e-2   &1.8e-3   &3.3e-5 \\
$r=15$			&  4.9e-1  &1.1e-1   &1.5e-2   &1.0e-3   &7.5e-5 \\
$r=20$			&  6.4e-1  &1.5e-1   &1.9e-2   &1.2e-3   &7.3e-5 \\ \hline
\end{tabular}
\normalsize
\end{center}
\caption{Final approximation errors (estimated by $\widetilde\gamma_{100}$) for different subsets $\Sc_X=\Rc_r(X)$ and different precisions $\delta$. Comparison with the theoretical upper bound $2\delta/(1-2\delta)$.}
\label{Tab2}
\end{table}

% \color{red}
% \begin{rmrk}
% The main task of a stopping criterion shall be considered later (?)
% \end{rmrk}
% \color{black}

%\textit{Numerical estimations of decreasing rate and 
%stagnation both show that the gradient-type algorithm 
%has better properties than predicted. One can improve 
%theoretical results with the following assumption : 
%there exist a constant $c$ that may depend on $u$ and 
%$\mathcal{S}_X$ such that $\Vert \Pi_{\mathcal{S}_X}(u+v)-u\Vert _X^2 
%\leq c^2 \Vert \Pi_{\mathcal{S}_X}(u)-u\Vert _X^2 + \Vert v\Vert _X^2$, 
%for all $v\in X$. Note that this assumption we make 
%here has been numericaly tested and appears to be true. 
%Then we have :
%\begin{align*}
% \Vert u^k-u\Vert _X^2 \leq \delta^{2k} \Vert u^0-u\Vert _X^2+ \frac{1-\delta^{2k}}{1-\delta^2}c^2\Vert u-\Pi_{\mathcal{S}_X}(u)\Vert _X^2
%\end{align*}
%Now, we recover a decreasing rate $\rho=\delta$, and 
%a stagnation that satisfies : $1+\text{limsup}_k \epsilon_k 
%\leq c/\sqrt{1-\delta^2}$. As we can see on figure 
%\ref{gradient_plateau} this majoration of the stagnation 
%is very accurate for $c=1$.
%}

%\begin{figure}[h]
%  \centering
%  \setlength\figwidth{8cm}
%  \setlength\figheight{4cm}
%  \include{courbes/gradient_plateau}
%  \caption{An accurate majoration of the stagnation.}
%  \label{gradient_plateau}
%\end{figure}

Now,  we focus on numerical estimations of the error  $\Vert u-u^k\Vert _X$.
It has been pointed out in Section \ref{sec:error_indicator} that $\hat \epsilon^k$, defined in Eq. \eqref{eq:error_estimator_projection}, should provide a good error estimator with effectivity index 
  $\hat \tau^k \in (1,(1-\delta^2)^{-1/2})$.
For $\delta=0.2$ and $\Sc_X=\Rc_{10}(X)$, numerical values taken by $\hat \tau^k$  during the gradient-type algorithm are plotted on Figure \ref{fig:eff_index} and are compared to the expected theoretical values of its lower and upper bounds $1$ and $(1-\delta^2)^{-1/2}$ respectively.
We observe that the theoretical upper bound is strictly satisfied, while the lower bound is almost but not exactly satisfied. 
This violation of the theoretical lower bound is explained by the fact that the precision $\delta$ is not satisfied at each iteration of the gradient-type algorithm due to the use of a heuristic convergence criterion in the computation of residuals  (see next section \ref{sec:construction_lambda}). {However, although it does not provide a controlled error estimation}, the error indicator based on the computed residuals is of very good quality. 

\begin{figure}[h]
  \centering
  \includegraphics[scale=0.9]{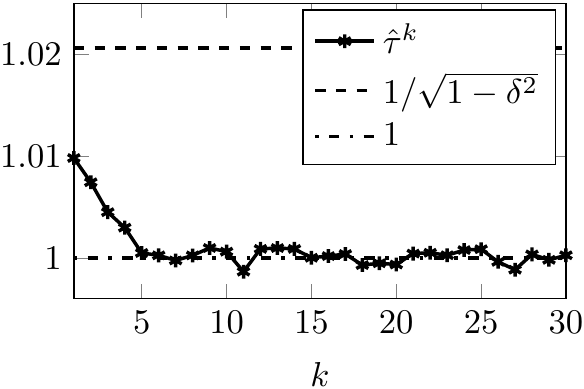}
  \caption{Effectivity index $\hat \tau^k$ of the error estimator $\hat\epsilon^{k} $ at different iterations $k$ of the gradient-type algorithm, with  $\mathcal{S}_X=\mathcal{R}_{10}(X)$ and $\delta=0.2$. \label{fig:eff_index}}
\end{figure}
%In the following numerical experiments, the gradient-type algorithm is stopped after the stagnation phase is reached.

\subsubsection{Application of $\Lambda^{\delta}$ for the approximation of residuals}\label{sec:construction_lambda}

We study the behavior of the updated  greedy algorithm described in Section  \ref{lambdadeltaalgo} for the computation of an approximation $y^{k}_{m}=\Lambda^{\delta}(r^{k})$ of the residual $r^{k}$ during the gradient-type algorithm. {Here, we use the particular strategy which consists in updating functions associated to each dimension $\mu\in I = \{1,2\}$ (steps (2)-(3) are performed two times per iteration)}. We first validate the ability of the heuristic stopping {criterion} \eqref{eq:stoppingcriterion} to ensure a prescribed relative precision $\delta$. Let $M=M(\delta)$ denote the iteration for which the condition $e_{M}^p\leq \delta$ is satisfied. The exact relative error $e_M = \Vert y^{k}_{M} - r^{k} \Vert_{Y}/\Vert r^{k}\Vert_{Y}$ is computed using a reference computation of $r^{k}$, and we define the effectivity index $\lambda^p_M=e_M^p/e_M$. Figure \ref{fig:gamma} shows the convergence of this effectivity index with respect to $p$, when using the natural canonical norm $\Vert\cdot\Vert_{2}$ or 
the weighted norm $\Vert\cdot\Vert_{w}$. We observe that $\lambda^p_M$ tends to $1$ as $p\to\infty$, as it was expected since the sequence  $\{y^k_m\}_{m\ge 1} $ converges to $r^k$. However, we clearly observe that the quality of the error indicator differs for the two different norms. 
When using the weighted norm, it appears that a large value of $p$ (say $p\geq20$) is necessary to ensure $\lambda^p_M \in [0.9,1]$, while $p\le 10$ seems sufficiently large when using the natural canonical norm. That simply reflects a slower convergence of the greedy algorithm when using the weighted norm. 

\begin{figure}[h]
  \centering
  \includegraphics[scale=0.9]{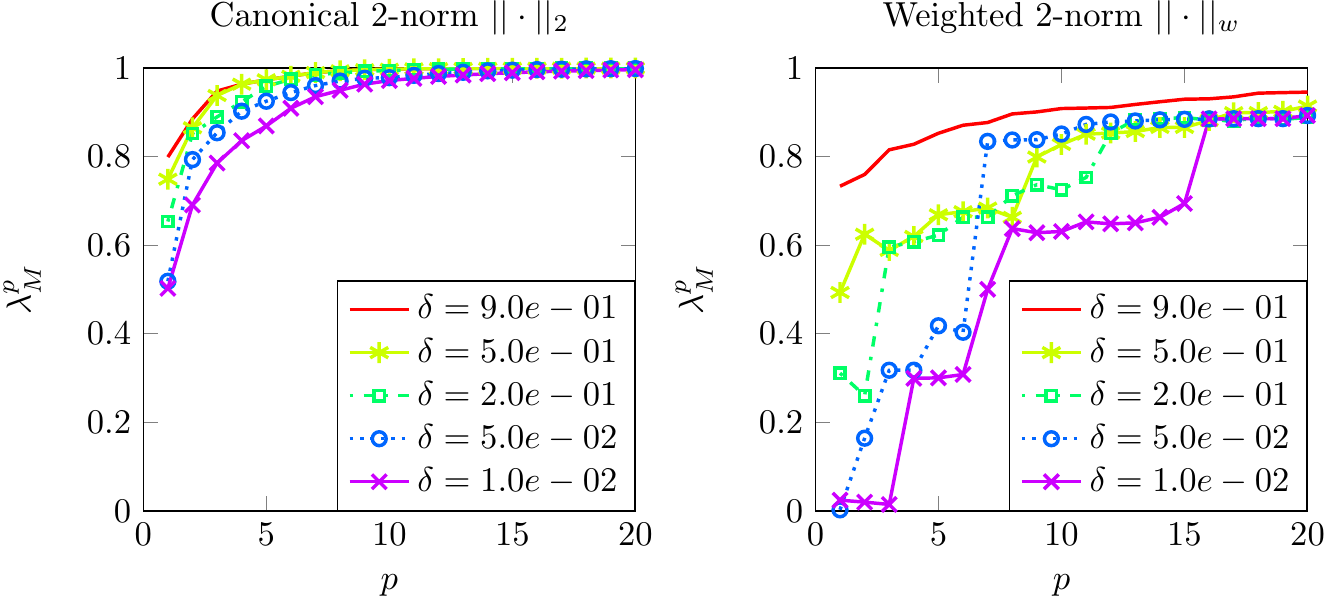}
  \caption{Evolution with $p$ of the effectivity index $\lambda^p_M$ for different $\delta$ at step $k=1$ of the gradient-type algorithm with $\Sc_{X}=\mathcal{R}_{10}(X)$ and for the natural canonical norm (left) or the weighted norm (right).}
  \label{fig:gamma}
\end{figure}

\begin{rmrk}
One can prove that at step $k$ of the gradient-type algorithm, when computing an approximation $y^{k}_{M}$ of $r^{k}$ with a greedy algorithm stopped using the heuristic stopping criterion
\eqref{eq:stoppingcriterion},  the effectivity index $\hat\tau^k$ of the computed error estimator $\hat \epsilon^k$ is such that 
$$
\hat\tau^k  \in \left(\sqrt{\frac{1-(\delta/\lambda^p_M)^2}{1-\delta^2}},\sqrt{\frac{1}{1-\delta^2}}\right).
$$
where $\lambda^p_M$ is the effectivity index of error indicator $e^{p}_{M}$ (supposed such that $\delta/\lambda^p_M<1$).
 That provides an explanation for the observations made on Figure \ref{fig:eff_index}.
\end{rmrk}

Now, we observe in Table \ref{Tab:ymXp} the number of iterations of the greedy algorithm for the approximation of the residual $r^{k}$ with a relative precision $\delta$, with a fixed value $p=20$ for the evaluation of the stopping criterion. The number of iterations corresponds to the rank of the resulting approximation. We note that the required rank is higher when using the weighted norm. 
It reflects the fact that it is more difficult to reach precision $\delta$ when using the weighted norm rather than the natural canonical norm.

\begin{table}[h]\begin{center}\footnotesize
\begin{tabular}{|c|c|c|c|c|c||c|c|c|c|c|}\hline
      ~ & \multicolumn{5}{c||}{Canonical 2-norm $\Vert \cdot\Vert _2$} & \multicolumn{5}{c|}{Weighted 2-norm $\Vert \cdot\Vert _w$} \\  
 k$\backslash\delta$  &0.9&0.5&0.2&0.05&0.01  &0.9 & 0.5  & 0.2  &0.05 & 0.01 \\ \hline
      1 & 1  &   1  &   3  &   7  &  11  &      8  &  21  &  31  &  35 &   51 \\
      2 & 1  &   3  &   7  &  16  &  27  &      5  &  22  &  14  &  24 &   42 \\
      3 & 1  &   5  &  11  &  19  &  24  &      4  &  15  &  24  &  23 &   43 \\
      4 & 1  &   3  &  11  &  14  &  24  &      8  &  11  &  19  &  37 &   42 \\
      5 & 1  &   6  &   7  &  15  &  24  &      6  &  19  &  23  &  14 &   38 \\
      6 & 1  &   8  &   8  &  16  &  24  &      3  &  12  &  47  &  25 &   63 \\
      7 & 1  &   5  &   7  &  17  &  24  &      7  &  14  &  16  &  29 &   47 \\
      8 & 1  &   4  &   8  &  16  &  24  &      5  &  12  &  22  &  21 &   40 \\
      9 & 1  &   4  &   8  &  16  &  24  &      7  &  13  &  18  &  36 &   45 \\
\hline
\end{tabular}\normalsize\end{center}
\caption{Computation of $\Lambda^{\delta}(r^{k})$ for different precisions $\delta$ and at different steps $k$ of the gradient-type algorithm, with $\Sc_X=\mathcal{R}_{10}$ (direct approach). The table indicates the number of greedy corrections computed  for reaching the precision $\delta$ using the heuristic stopping {criterion} \eqref{eq:stoppingcriterion} with $p=20$. }
\label{Tab:ymXp}\end{table}

\subsection{ {Higher dimensional case} }

%Now, we consider the solution of \eqref{eq::addiffre} with a stochastic dimension of 9. To this goal, 
Now, we consider a diffusion coefficient of the form $\kappa(x,\xi) = \kappa_0 + \sum_{i=1}^8 \xi_i \kappa_i(x)$ where $\kappa_0=10$, $\xi_i \sim U(-1,1)$ are independent uniform random variables, and the functions $\kappa_i(x)$ are given by:
\begin{equation*}
 \begin{array}{cccc}
  \kappa_1(x)=\cos(\pi x_1), & \kappa_3(x)=\sin(\pi x_1), & \kappa_5(x)=\cos(\pi x_1)\cos(\pi x_2), & \kappa_7(x)=\cos(\pi x_1)\sin(\pi x_2), \\
  \kappa_2(x)=\cos(\pi x_2), & \kappa_4(x)=\sin(\pi x_2), & \kappa_6(x)=\sin(\pi x_1)\sin(\pi x_2), & \kappa_8(x)=\sin(\pi x_1)\cos(\pi x_2).
 \end{array}
\end{equation*}
In addition, the advection coefficient is given by $c=\xi_0 c_0$, where $\xi_0 \sim U(0,4000)$ is a uniform random variable. 
% The convection coefficient $c_0$ as the source term $f$ are the same as in the previous example. Finally we suppose $a=0$. \\
We denote ${V}=\xHone_{0}(\Omega)$ and ${S}=\xLtwo(\Xi,dP_{\xi})$ where $(\Xi,\mathcal{B}(\Xi),P_\xi)$ is a probability space with $\Xi=]-1,1[^8 \times ]0,4000[$ and $P_\xi$ the uniform measure. Here $V_N \subset V $ is a $\mathbb{Q}_1$ finite element space associated with a uniform mesh of 3600 elements, with a dimension $N=3481$. We take $S_P = \otimes_{i=0}^8 S_P^{\xi_i} \subset S$, where $S_P^{\xi_i}$ are polynomial function spaces of degree $7$ on $\Xi_i$ with $P=\dim( S_P^{\xi_i})=8$. Then, the Galerkin approximation in $V_N\otimes S_P$ (solution of \eqref{eq:FV}) is searched under the form
$u=\sum_{i=1}^N \sum_{j_0=1}^P\cdots\sum_{j_8=1}^P (u_{i,j_0,\cdots,j_9}) \phi_j\otimes(\otimes_{\mu=0}^8 \psi_{j_\mu}^\mu )$. This Galerkin approximation can be identified with its set of coefficients, still denoted by $u$ which is a tensor
\begin{equation}
u\in X = \mathbb{R}^N\otimes(\otimes_{\mu=0}^8\mathbb{R}^P)\quad
\text{ such that } \quad Au=b ,
\label{eq:alg2}
\end{equation}
 where $A$ and $b$ are the algebraic representations on the chosen basis of $V_N\otimes S_P$ of the bilinear and linear forms in \eqref{eq:FV}. The obtained algebraic system of equations has a dimension larger than $10^{11}$ and its solution clearly requires the use of model {reduction methods}.\\

Here, we compute low rank approximations of the solution of \eqref{eq:alg2} in the canonical tensor subset $\mathcal{R}_r(X)$ with $r\ge 1$. Since best approximation problems in  $\mathcal{R}_r(X)$ are well posed for $r=1$ but ill posed for $d>2$ and $r>1$, we rely on the greedy algorithm presented in section   \ref{sec:greedy} with successive corrections in $\Sc_X=\mathcal{R}_1(X)$ computed with Algorithm \ref{alg:procedure1}. 
\begin{rmrk}
Low-rank approximations could have been computed directly with Algorithm \ref{alg:procedure1} by choosing for $\Sc_X$ other stable low-rank formats adapted to high-dimensional problems, such as Hierarchical Tucker (or Tensor Train) low-rank formats.
\end{rmrk}

\subsubsection{Convergence study}

In this section, low rank approximations of the solution $u$ of \eqref{eq:alg2} are computed for the two different norms $\Vert\cdot\Vert_2$ and $\Vert\cdot\Vert_w$ defined as in section \ref{sec:comp_minres}. Here, we assume that the weighting function $w$ is equal to $100$ in the subdomain $D\subset\Omega$, and $1$ elsewhere.\\
Since $\text{dim}(X)\ge 10^{11}$, the exact Galerkin approximation $u$ in $X$ is no more computable. As a reference solution, we consider a low-rank  approximation $u_{\text{ref}}$ of $u$  computed using a greedy rank-one algorithm based on a canonical minimal residual formulation. {We introduce an estimation  $\hat E_K$  of $\frac{\Vert u-u_{\text{ref}}\Vert_2}{\Vert u\Vert_2}$  based on Monte-Carlo integrations using $K$ realizations $\{\xi_k\}_{k=1}^K$ of the random variable $\xi$, defined by 
$$
\hat E_K^2 =  \frac{\frac{1}{K} \sum_{k=1}^K \Vert u(\xi_k)-u_{\text{ref}}(\xi_k)\Vert_V^2}{\frac{1}{K}\sum_{k=1}^K \Vert u(\xi_k)\Vert_V^2},
$$
with a number of samples $K$ such that the Monte-Carlo estimates has a relative standard deviation (estimated using the statistical variance of the sample) lower than $10^{-1}$.
The rank of $u_{\text{ref}}$ is here selected such that $ \hat E_K <10^{-4}$, which gives a reference solution with a rank of $212$.
}
% {Moreover, for any $\widetilde u \in X$,
%$$
%\Big| \Vert \widetilde u-u_{\text{ref}}\Vert_V - \Vert \widetilde u-u \Vert_V \Big|\leq \Vert u-u_{\text{ref}}\Vert_V \approx \epsilon \Vert u \Vert_V
%$$
%implies that $\Vert \widetilde u-u_{\text{ref}}\Vert_V/\Vert u_{\text{ref}} \Vert_V$ provides a good estimation of $\Vert \widetilde u-u \Vert_V /\Vert u \Vert_V$, at least for $\Vert \widetilde u-u_{\text{ref}}\Vert_V/\Vert u_{\text{ref}} \Vert_V \gg \epsilon$.\\
%}

{On Figure \ref{HD_convergence_norm}, we plot the convergence with the rank $r$ of the approximations computed by both A-IMR and CMR algorithms and of the greedy approximations $\widetilde{u}^r_2$ and $\widetilde{u}^r_w$ of the reference solution $u_{\text{ref}}$ for both norms.}
We observe (as for the lower-dimensional example) that for both norms, with different values of the parameter $\delta$ (up to $0.9$), the A-IMR method provides a  better approximation of the solution in comparison to the CMR method. 
{When decreasing $\delta$, the proposed algorithm seems to provide approximations that tend to present the same convergence as the greedy approximations $\widetilde{u}^r_2$ and $\widetilde{u}^r_w$. }

\begin{figure}[h]
   \centering
   \subfigure[Canonical norm]{
     \centering
     \includegraphics[scale=0.75]{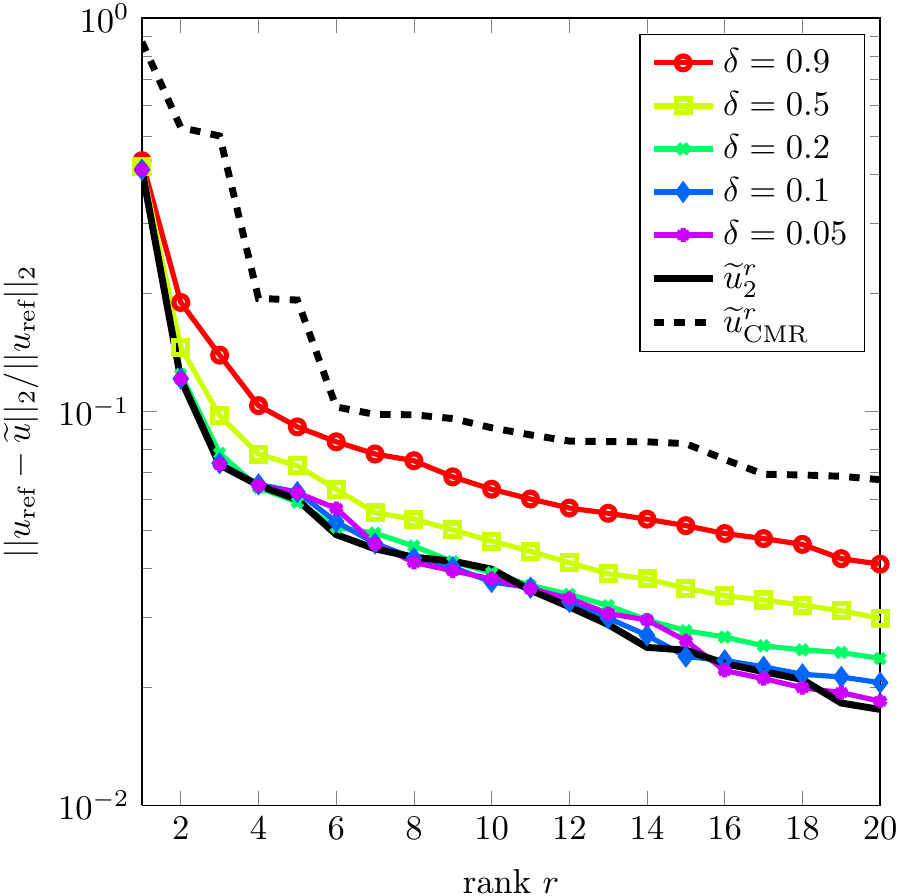}
     }
   \subfigure[Weighted norm]{
     \centering
     \includegraphics[scale=0.75]{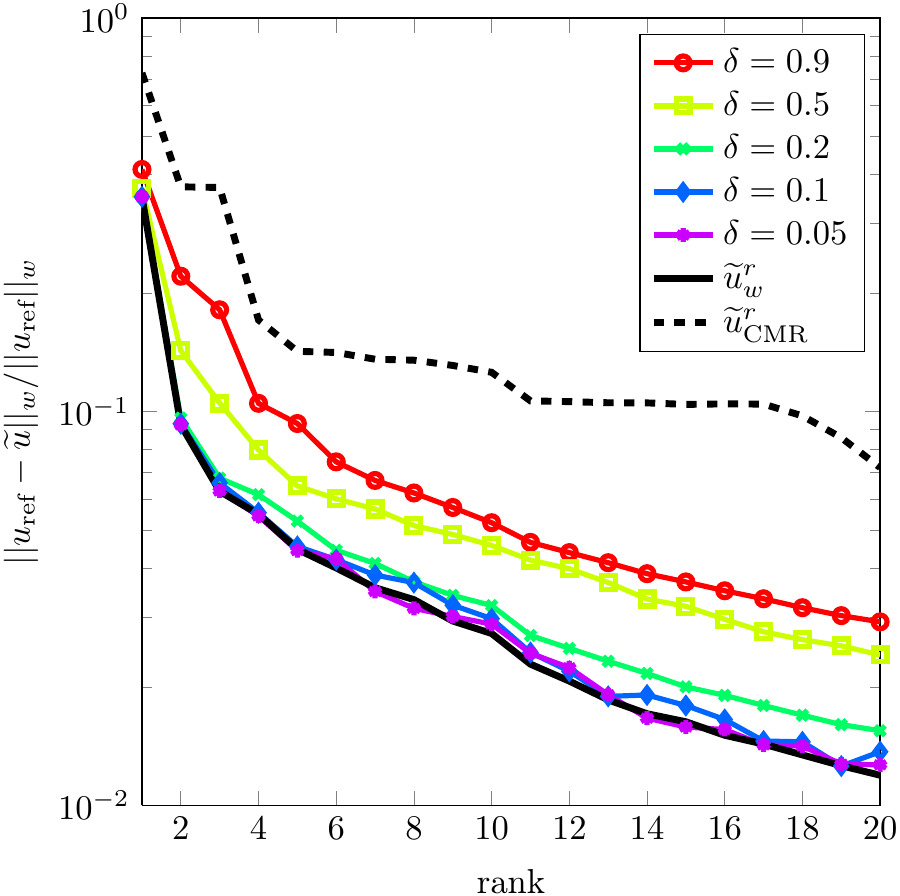}
     }
   \caption{Convergence with the rank of approximations obtained with the greedy CMR or A-IMR algorithms for different precisions $\delta$. On the left (resp. right) plot, convergence is plotted with respect to the norm $\|\cdot\|_2$ (resp. $\| \cdot \|_w$) and A-IMR is used with the objective norm  $\|\cdot\|_2$ (resp. $\| \cdot \|_w$). }  \label{HD_convergence_norm}
\end{figure}

\subsubsection{Study of the greedy algorithm for $\Lambda^\delta$}

Now,  we study the behavior of the updated  greedy algorithm described in Section  \ref{lambdadeltaalgo} for the computation of an approximation $y^{k}_{m}=\Lambda^{\delta}(r^{k})$ of the residual $r^{k}$ during the gradient-type algorithm. Here, we use the particular strategy which consists in updating functions associated to each dimension $\mu\in I = \{2,\hdots,10\}$ (steps (2)-(3) are performed 9 times per iteration). The update of functions associated with the first dimension is not performed since it would involve the expensive computation of approximations in a space $Z^k_m$ with a large dimension $m N$. \\

In table \ref{Tab:ymXp_HD}, we summarize  the required number of  greedy corrections needed at each iteration of the gradient type algorithm for reaching the precision $\delta$ with the heuristic stagnation criterion \eqref{eq:stoppingcriterion} with $p=20$. As for the previous lower-dimensional test case, the number of corrections increases as $\delta$ decreases and is higher for the weighted norm than for the canonical norm. However, we observe that this number of corrections remains reasonable even for small $\delta$. 
 
\begin{table}[h]\begin{center}\footnotesize
\begin{tabular}{|c|c|c|c|c|c||c|c|c|c|c|}\hline
      ~ & \multicolumn{5}{c||}{Canonical 2-norm $\Vert \cdot\Vert _2$} & \multicolumn{5}{c|}{Weighted 2-norm $\Vert \cdot\Vert _w$} \\  
 k$\backslash\delta$  &0.9&0.5&0.2&0.05&0.01  &0.9 & 0.5  & 0.2  &0.05 & 0.01 \\ \hline
1 &  1 &  1 &  3 &  6 & 14 &  3 & 12 & 53 & 65 & 91 \\
2 &  1 &  3 &  5 & 13 & 24 &  2 & 11 & 49 & 62 & 91 \\
3 &  1 &  3 &  5 & 12 & 17 &  3 & 12 & 49 & 62 & 91 \\
4 &  1 &  3 &  5 & 13 & 26 &  3 & 12 & 53 & 62 & 91 \\
5 &  1 &  3 &  6 & 12 & 24 &  2 & 11 & 47 & 65 & 89 \\
6 &  1 &  3 &  5 & 13 & 27 &  3 & 11 & 42 & 63 & 88 \\
7 &  1 &  3 &  5 & 12 & 27 &  3 & 10 & 50 & 65 & 88 \\
8 &  1 &  3 &  5 & 12 & 26 &  3 & 10 & 49 & 60 & 87 \\
9 &  1 &  3 &  6 & 12 & 26 &  3 & 13 & 49 & 65 & 80 \\ \hline
\end{tabular}\normalsize\end{center}
\caption{Computation of $\Lambda^{\delta}(r^{k})$ for different precisions $\delta$ and at different steps $k$ of the gradient-type algorithm (first iteration $r=1$ of the greedy approach with $\Sc_X=\mathcal{R}_{1}$). The table indicates the number of greedy corrections computed  for reaching the precision $\delta$ using the heuristic stopping {criterion} \eqref{eq:stoppingcriterion} with $p=20$.}
\label{Tab:ymXp_HD}\end{table}

\subsubsection{Estimation of a quantity of interest}

Finally, we study the quality of the low rank approximations $\widetilde u$ obtained with both CMR and A-IMR algorithms for the canonical and weighted norms. To this end, we compute the quantity of interest $Q(\widetilde{u})$ defined by \eqref{eq:QOIdef}.
Figure \ref{fig:HD_convergence_variance} illustrates the convergence with the rank of the variance of the approximate quantities of interest.
{
Note that the algorithm do not guarantee the monotone convergence of the quantity of interest with respect to the rank, that is confirmed by the numerical results. 
% The same observation can be done on the Figure \ref{fig:convergence_Q}: it seems that in the higher dimensional case, this phenomena is more important.
}
However, we observe that the approximations provided by the A-IMR algorithm are better than the ones given by the CMR, even for large $\delta$.
{Also, when using the weighted norm in the A-IMR algorithm, the quantity of interest is estimated with an better precision.}
Similar behaviors are observed for the convergence of the mean.

\begin{figure}[h]
   \centering
   \subfigure[Using the canonical norm.]{
     \centering
     \includegraphics[scale=0.75]{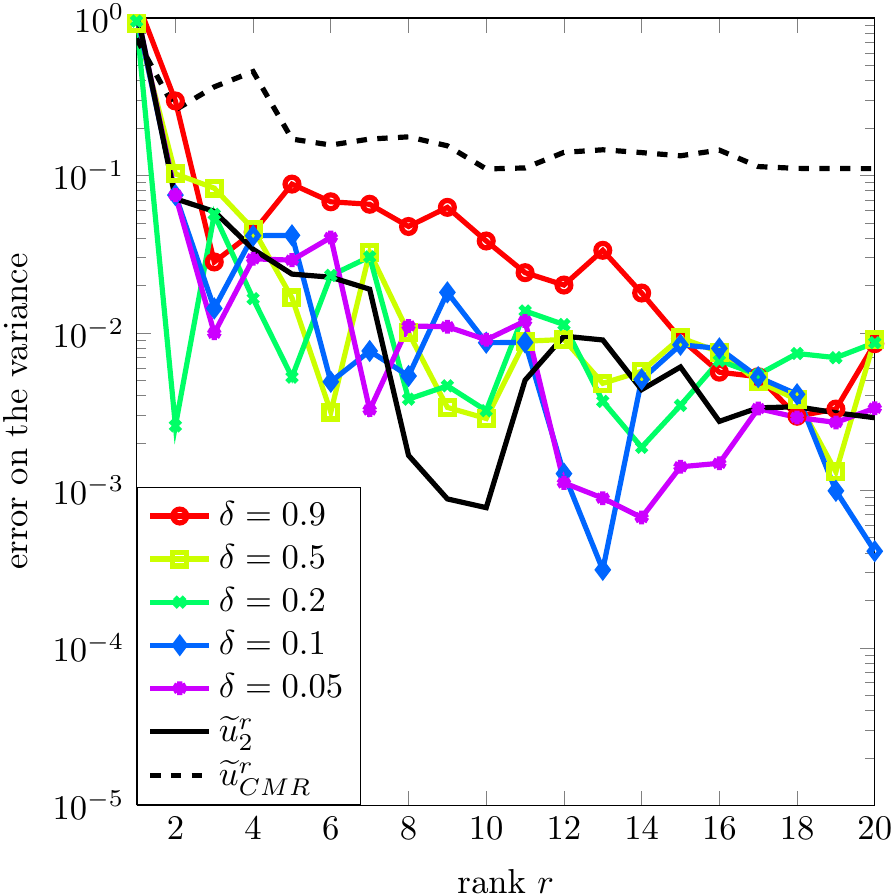}
     }
   \subfigure[Using the weighted norm.]{
     \centering
     \includegraphics[scale=0.75]{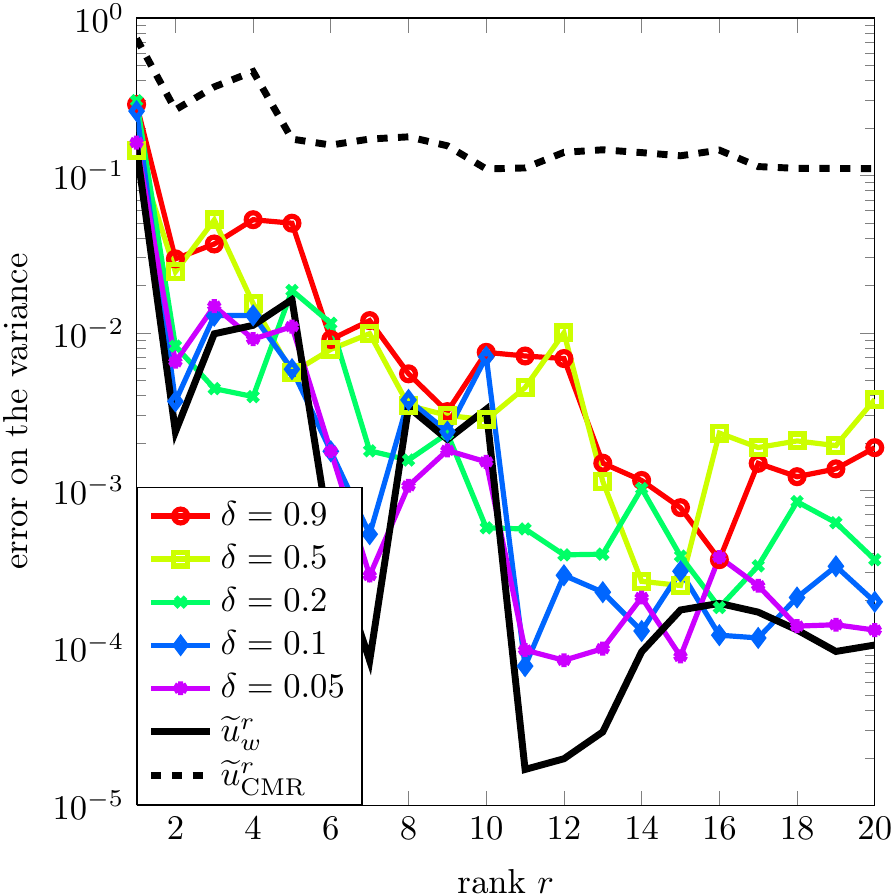}
     }
   \caption{Relative error with respect to the variance of  the reference solution 
   $Q(u_{\text{ref}})$  with the canonical (left) and weighted  (right) norms.}
   \label{fig:HD_convergence_variance}
\end{figure}

\section{Conclusion}
In this paper, we have proposed a new algorithm for the construction of low-rank approximations of the solutions of high-dimensional weakly coercive problems formulated  in a tensor space $X$. 
This algorithm is based on the approximate minimization (with a certain precision $\delta$) of a particular residual norm on given low-rank tensor subsets $\Sc_{X}$, the residual norm coinciding with some measure of the error in solution. Therefore, the algorithm is able to provide a quasi-best low-rank approximation with respect to a norm $\Vert\cdot\Vert_{X}$ that can be designed for a certain objective.
A weak greedy algorithm using this minimal residual approach has been introduced and its convergence has been proved under some conditions. A numerical example dealing with the solution of a stochastic partial differential equation has illustrated the effectivity of the method and the properties of the proposed algorithms.
Some technical points  have to be addressed in order to apply the method to a more general setting and to improve its efficiency {and robustness}: the development of efficient solution methods for the computation of residuals when using general norms $\Vert\cdot\Vert_{X}$ (that are not induced norms in the tensor space $X$), {the introduction of robust error estimators  during the computation of residuals (for the robust control of the precision $\delta$, {which is the key point for controlling the quality of the obtained approximations})}, the application of the method for using tensor formats adapted to high-dimensional problems (such as Hierarchical formats). Also, a challenging perspective consists in coupling low-rank approximation techniques with adaptive approximations in infinite-dimensional tensor spaces {(as in \cite{Bachmayr2013})} in order to provide approximations of high-dimensional equations (PDEs or stochastic PDEs) with a complete control on the precision of quantities of interest.


\begin{thebibliography}{}

\bibitem{Ammar2006}
A.~Ammar, B.~Mokdad, F.~Chinesta, and R.~Keunings.
\newblock A new family of solvers for some classes of multidimensional partial
  differential equations encountered in kinetic theory modelling of complex
  fluids.
\newblock {\em Journal of Non-Newtonian Fluid Mechanics}, 139(3):153--176,
  2006.

\bibitem{Ammar2010}
A.~Ammar, F.~Chinesta, and A.~Falco.
\newblock On the convergence of a greedy rank-one update algorithm for a class
  of linear systems.
\newblock {\em Archives of Computational Methods in Engineering},
  17(4):473--486, 2010.

\bibitem{Bachmayr2013}
 M. Bachmayr,  and W. Dahmen.
 \newblock Adaptive Near-Optimal Rank Tensor Approximation for High-Dimensional Operator Equations
\newblock{\em ArXiv e-print}, November 2013.

\bibitem{Ballani2013}
J. Ballani and L. Grasedyck.
\newblock A projection method to solve linear systems in tensor format.
\newblock {\em Numerical Linear Algebra with Applications}, 20(1):27-43, 2013.

\bibitem{Beylkin2005}
G.~Beylkin and M.~J. Mohlenkamp.
\newblock Algorithms for numerical analysis in high dimensions.
\newblock {\em SIAM J. Sci. Comput.}, 26(6):2133--2159, 2005.

\bibitem{Cances2011}
E.~Cances, V.~Ehrlacher, and T.~Lelievre.
\newblock Convergence of a greedy algorithm for high-dimensional convex
  nonlinear problems.
\newblock {\em Mathematical Models \& Methods In Applied Sciences},
  21(12):2433--2467, 2011.

\bibitem{Cances2012}
E.~{Cances}, V.~{Ehrlacher}, and T.~{Lelievre}.
\newblock {Greedy algorithms for high-dimensional non-symmetric linear
  problems}.
\newblock {\em ArXiv e-prints}, October 2012.

\bibitem{Cohen2012}
A. Cohen, W. Dahmen, and G. Welper.
\newblock Adaptivity and variational stabilization for convection-diffusion
  equations.
\newblock {\em ESAIM: Mathematical Modelling and Numerical Analysis},
  46:1247--1273, 8 2012.


\bibitem{Chinesta2011}
F.~Chinesta, P.~Ladeveze, and E.~Cueto.
\newblock A short review on model order reduction based on proper generalized
  decomposition.
\newblock {\em Archives of Computational Methods in Engineering},
  18(4):395--404, 2011.


\bibitem{Dahmen2012}
W.~Dahmen, C.~Huang, C.~Schwab, and G.~Welper.
\newblock Adaptive petrov--galerkin methods for first order transport
  equations.
\newblock {\em SIAM Journal on Numerical Analysis}, 50(5):2420--2445, 2012.



\bibitem{DeLathauwer2000}
L.~{De Lathauwer}, B.~{De Moor}, and J.~Vandewalle.
\newblock A multilinear singular value decomposition.
\newblock {\em SIAM J. Matrix Anal. Appl.}, 21(4):1253--1278, 2000.

\bibitem{Doostan2009}
A.~Doostan and G.~Iaccarino.
\newblock A least-squares approximation of partial differential equations with
  high-dimensional random inputs.
\newblock {\em Journal of Computational Physics}, 228(12):4332--4345, 2009.


\bibitem{Ern2004}
A.~Ern and J.-L. Guermond.
\newblock Theory and practice of finite elements, volume 159 of applied
  mathematical sciences, 2004.

\bibitem{Espig2012}
M. Espig and W. Hackbusch.
\newblock A regularized newton method for the efficient approximation of
  tensors represented in the canonical tensor format.
\newblock {\em Numerische Mathematik}, 122:489--525, 2012.


\bibitem{Falco2011}
A.~Falc{\'o} and A.~Nouy.
\newblock A {Proper Generalized Decomposition} for the solution of elliptic
  problems in abstract form by using a functional {Eckart-Young} approach.
\newblock {\em Journal of Mathematical Analysis and Applications},
  376(2):469--480, 2011.
  
\bibitem{Falco2012-minimal}
A. Falc{\'o} and W. Hackbusch.
\newblock On minimal subspaces in tensor representations.
\newblock {\em Foundations of Computational Mathematics}, 12:765--803, 2012.


\bibitem{Falco2012-pgd}
A. Falc{\'o} and A. Nouy.
\newblock Proper generalized decomposition for nonlinear convex problems in
  tensor banach spaces.
\newblock {\em Numerische Mathematik}, 121:503--530, 2012.


\bibitem{Falco2013}
A. Falc{\'o}, W. Hackbusch, and A. Nouy.
\newblock Geometric structures in tensor representations. Preprint 9/2013, MPI MIS.
%\newblock {\em Submitted to Found. Comput. Math.}.

\bibitem{Figueroa2012}
L.~Figueroa and E.~Suli.
\newblock Greedy approximation of high-dimensional {Ornstein}-{Uhlenbeck}
  operators.
\newblock {\em Foundations of Computational Mathematics}, 12:573--623, 2012.

\bibitem{Giraldi2012}
L.~Giraldi.
\newblock {\em Contributions aux M\'ethodes de Calcul Bas\'ees sur
  l'Approximation de Tenseurs et Applications en M\'ecanique Num\'erique.}
\newblock PhD thesis, Ecole Centrale Nantes, 2012.

\bibitem{Giraldi2013}
L.~Giraldi, A.~Nouy, G.~Legrain, and P.~Cartraud.
\newblock Tensor-based methods for numerical homogenization from
  high-resolution images.
\newblock {\em Computer Methods in Applied Mechanics and Engineering},
  254(0):154 -- 169, 2013.

\bibitem{Grasedyck2010}
L.~Grasedyck.
\newblock Hierarchical singular value decomposition of tensors.
\newblock {\em SIAM J. Matrix Anal. Appl.}, 31:2029--2054, 2010.

\bibitem{Grasedyck2013}
L.~Grasedyck, D.~Kressner, and C.~Tobler.
\newblock A literature survey of low-rank tensor approximation techniques.
\newblock {\em GAMM-Mitteilungen}, 2013.


\bibitem{Hackbusch2012}
W.~Hackbusch.
\newblock {Tensor Spaces and Numerical Tensor Calculus},
\newblock {\em Series in Computational Mathematics}
 volume 42. Springer, 2012. 

\bibitem{Hackbusch2009}
W.~Hackbusch and S.~Kuhn.
\newblock {A New Scheme for the Tensor Representation}.
\newblock {\em Journal of Fourier analysis and applications},
  {15}({5}):{706--722}, {2009}.

\bibitem{Holtz2012ALS}
S.~Holtz, T.~Rohwedder, and R.~Schneider.
\newblock {The Alternating Linear Scheme for Tensor Optimisation in the TT format}.
\newblock {\em SIAM Journal on Scientific Computing},
 34(2):683--713, 2012.

\bibitem{Holtz2012manifolds}
S.~Holtz, T.~Rohwedder, and R.~Schneider.
\newblock{On manifolds of tensors with fixed TT rank}.
\newblock{\em Numer. Math.},
 120(4):701--731, 2012.

\bibitem{Khoromskij2011} B.~N.~Khoromskij and C.~Schwab.
\newblock {Tensor-structured Galerkin approximation of parametric and stochastic elliptic PDEs}.
\newblock {\em SIAM Journal on Scientific Computing}, 33(1):364--385, 2011.

\bibitem{Khoromskij2012}
B. N. Khoromskij.
\newblock {Tensors-structured numerical methods in scientific computing: Survey on recent advances.}
\newblock {\em Chemometrics and Intelligent Laboratory Systems}, 110(1):1--19, 2012.

\bibitem{Kolda2009}
T.~G. Kolda and B.~W. Bader.
\newblock Tensor decompositions and applications.
\newblock {\em SIAM Review}, 51(3):455--500, September 2009.

\bibitem{Kressner2011}
D.~Kressner and C.~Tobler.
\newblock Low-rank tensor krylov subspace methods for parametrized linear
  systems.
\newblock {\em SIAM Journal on Matrix Analysis and Applications},
  32(4):1288--1316, 2011.

\bibitem{Ladeveze1999}
P.~Ladev\`eze.
\newblock {\em Nonlinear Computational Structural Mechanics - New Approaches
  and Non-Incremental Methods of Calculation}.
\newblock Springer Verlag, 1999.


\bibitem{Ladeveze2010}
P.~Ladev\`eze, J.C. Passieux, and D.~N\'eron.
\newblock The {LATIN} multiscale computational method and the {Proper}
  {Generalized} {Decomposition}.
\newblock {\em Computer Methods in Applied Mechanics and Engineering},
  199(21-22):1287--1296, 2010.



\bibitem{Matthies2012}
H.~G. Matthies and E.~Zander.
\newblock Solving stochastic systems with low-rank tensor compression.
\newblock {\em Linear Algebra and its Applications}, 436(10), 2012.


\bibitem{Nouy2007}
A.~Nouy.
\newblock {A generalized spectral decomposition technique to solve a class of linear stochastic partial differential equations},
\newblock {\em Computer Methods in Applied Mechanics and Engineering}, 196(45-48) (2007) 4521-4537.

\bibitem{Nouy2009} A. Nouy, Recent developments in spectral stochastic methods for the numerical solution of stochastic partial differential equations,
Archives of Computational Methods in Engineering, 16(3) (2009) 251-285.

\bibitem{Nouy2010}
A.~Nouy.
\newblock {Proper Generalized Decompositions} and separated representations for
  the numerical solution of high dimensional stochastic problems.
\newblock {\em Archives of Computational Methods in Engineering},
  17(4):403--434, 2010.
  
  \bibitem{Nouy2010-time}
A.~Nouy.
\newblock A priori model reduction through proper generalized decomposition for
  solving time-dependent partial differential equations.
\newblock {\em Computer Methods in Applied Mechanics and Engineering},
  199(23-24):1603--1626, 2010.

\bibitem{Oseledets2009}
I.~V.~Oseledets and E.~E.~Tyrtyshnikov.
\newblock Breaking the curse of dimensionality, or how to use {SVD} in many
  dimensions.
\newblock {\em SIAM J. Sci. Comput.}, 31(5):3744--3759, 2009.

\bibitem{Oseledets2011}
I.~V.~Oseledets.
\newblock Tensor-train decomposition.
\newblock {\em SIAM J. Sci. Comput.}, 33(5):2295--2317, 2011.

\bibitem{Rohwedder2013}
T.~Rohwedder and A.~Uschmajew.
\newblock On local convergence of alternating schemes for optimization of
  convex problems in the tensor train format.
\newblock {\em SIAM J. Numer. Anal.}, 51(2):1134-1162, 2013.

\bibitem{Temlyakov2011}
V.~Temlyakov.
\newblock {\em Greedy Approximation}.
\newblock Cambridge Monographs on Applied and Computational Mathematics.
  Cambridge University Press, 2011.

\bibitem{Temlyakov2008}
V. Temlyakov.
\newblock Greedy approximation.
\newblock {\em Acta Numerica}, 17:235--409, 2008.

\bibitem{Uschmajew2012}
A.~Uschmajew and B.~Vandereycken.
\newblock The geometry of algorithms using hierarchical tensors.
\newblock Technical report, ANCHP-MATHICSE, Mathematics Section, EPFL, 2012.
\end{thebibliography}
\end{document}